\documentclass{amsart}

\ifx\Presentation\undefined
\usepackage[breaklinks]{hyperref}
\fi
\usepackage{amsmath}
\usepackage{amssymb}
\usepackage{mathrsfs}
\usepackage{amsthm}
\usepackage[utf8]{inputenc}
\usepackage{svninfo}
\usepackage{prettyref}
\ifx\FrenchText\undefined
\usepackage[british]{babel}
\else
\usepackage[british,french]{babel}
\AtBeginDocument{%
\catcode`\:=12%
\catcode`\!=12%
}
\fi

\makeatletter

\ifx\Presentation\undefined
\fi

\newcommand{\fref}[1]{\prettyref{#1}}
\newrefformat{item}{\ref{#1}}
\newrefformat{apx}{Appendix~\ref{#1}}
\newrefformat{cha}{Chapter~\ref{#1}}
\newrefformat{sec}{Section~\ref{#1}}

\newcounter{dummycnt}

\ifx\thmnum\undefined
\newtheorem{thmnum}{}[section]
\fi

\newcommand{\mynewthm}[3][]{%
  \def\PARAM{#1}
  \ifx\PARAM\empty
  \newtheorem{#2}[thmnum]{#3}
  \else
  \newtheorem{#2}{#3}[#1]
  \fi
  \newtheorem*{#2*}{#3}%
  \newrefformat{#2}{#3~\ref{##1}}%
}

\ifx\FrenchText\undefined
\newcommand{\ThmLabel}{Theorem}
\newcommand{\PrpLabel}{Proposition}
\newcommand{\LemLabel}{Lemma}
\newcommand{\FctLabel}{Fact}
\newcommand{\CorLabel}{Corollary}
\newcommand{\DfnLabel}{Definition}
\newcommand{\ConvLabel}{Convention}
\newcommand{\NtnLabel}{Notation}
\newcommand{\CstLabel}{Construction}
\newcommand{\ExmLabel}{Example}
\newcommand{\RmkLabel}{Remark}
\newcommand{\QstLabel}{Question}

\else
\newcommand{\ThmLabel}{\iflanguage{french}{Théorème}{Theorem}}
\newcommand{\PrpLabel}{Proposition}
\newcommand{\LemLabel}{\iflanguage{french}{Lemme}{Lemma}}
\newcommand{\FctLabel}{\iflanguage{french}{Fait}{Fact}}
\newcommand{\CorLabel}{\iflanguage{french}{Corollaire}{Corollary}}
\newcommand{\DfnLabel}{\iflanguage{french}{Définition}{Definition}}
\newcommand{\ConvLabel}{Convention}
\newcommand{\NtnLabel}{Notation}
\newcommand{\CstLabel}{Construction}
\newcommand{\ExmLabel}{\iflanguage{french}{Exemple}{Example}}
\newcommand{\RmkLabel}{\iflanguage{french}{Remarque}{Remark}}
\newcommand{\QstLabel}{Question}

\fi

\theoremstyle{plain}
\mynewthm{thm}{\ThmLabel}
\mynewthm{prp}{\PrpLabel}
\mynewthm{lem}{\LemLabel}
\mynewthm{fct}{\FctLabel}
\mynewthm{cor}{\CorLabel}

\theoremstyle{definition}
\mynewthm{dfn}{\DfnLabel}
\mynewthm{conv}{\ConvLabel}
\mynewthm{conj}{Conjecture}
\mynewthm{ntn}{\NtnLabel}
\mynewthm{cst}{\CstLabel}

\theoremstyle{remark}
\mynewthm{rmk}{\RmkLabel}
\mynewthm{qst}{\QstLabel}
\mynewthm{exm}{\ExmLabel}

\ifx\Presentation\undefined

\newcommand{\myenumlabel}[1]{\textnormal{(\roman{#1})}}
\renewcommand{\theenumi}{\myenumlabel{enumi}}
\fi

\newcounter{cycprfcnt}
\newenvironment{cycprf}%
{\begin{list}{\PackageWarning{begnac}{Label required for cycprf}}%
  {%
    \setcounter{cycprfcnt}{1}
    \setlength{\itemindent}{0.5\leftmargin}%
    \setlength{\leftmargin}{0pt}%
    \newcommand{\cpcurr}{\myenumlabel{cycprfcnt}}%
    \newcommand{\cpnext}{\addtocounter{cycprfcnt}{1}\cpcurr}%
    \newcommand{\cpprev}{\addtocounter{cycprfcnt}{-1}\cpcurr}%
    \newcommand{\cpnum}[1]{\setcounter{cycprfcnt}{##1}\cpcurr}%
    \newcommand{\cpfirst}{\cpnum{1}}%
    \newcommand{\impnext}{\cpcurr{} $\Longrightarrow$ \cpnext.}%
    \newcommand{\impfirst}{\cpcurr{} $\Longrightarrow$ \cpfirst.}%
  }%
}%
{\qedhere\end{list}}%

\def\indsym#1#2{%
  \setbox0=\hbox{$\m@th#1x$}%
  \kern\wd0%
  \hbox to 0pt{\hss$\m@th#1\mid$\hbox to 0pt{$\m@th#1^{#2}$\hss}\hss}%
  \lower.9\ht0\hbox to 0pt{\hss$\m@th#1\smile$\hss}%
  \kern\wd0}
\newcommand{\ind}[1][]{\mathop{\mathpalette\indsym{#1}}}

\def\nindsym#1#2{%
  \setbox0=\hbox{$\m@th#1x$}%
  \kern\wd0%
  \hbox to 0pt{\hss$\m@th#1\not$\kern1.4\wd0\hss}
  \hbox to 0pt{\hss$\m@th#1\mid$\hbox to 0pt{$\m@th#1^{#2}$\hss}\hss}%
  \lower.9\ht0\hbox to 0pt{\hss$\m@th#1\smile$\hss}%
  \kern\wd0}
\newcommand{\nind}[1][]{\mathop{\mathpalette\nindsym{#1}}}

\def\dotminussym#1#2{%
  \setbox0=\hbox{$\m@th#1-$}%
  \kern.5\wd0%
  \hbox to 0pt{\hss\hbox{$\m@th#1-$}\hss}%
  \raise.6\ht0\hbox to 0pt{\hss$\m@th#1.$\hss}%
  \kern.5\wd0}
\newcommand{\dotminus}{\mathbin{\mathpalette\dotminussym{}}}

\renewcommand{\emptyset}{\varnothing}
\renewcommand{\setminus}{\smallsetminus}
\def\models{\vDash}

\newcommand{\concat}{{^\frown}}
\newcommand{\rest}{{\restriction}}

\newcommand{\sfrac}[2]{\hbox{$\frac{#1}{#2}$}}
\newcommand{\half}[1][1]{\sfrac{#1}{2}}

\DeclareMathOperator{\tp}{tp}

\newcommand{\Cb}{\mathrm{Cb}}
\DeclareMathOperator{\Th}{Th}

\DeclareMathOperator{\tS}{S}
\DeclareMathOperator{\dcl}{dcl}
\DeclareMathOperator{\acl}{acl}

\DeclareMathOperator{\SU}{SU}

\DeclareMathOperator{\id}{id}

\DeclareMathOperator{\Aut}{Aut}

\DeclareMathOperator{\Diag}{Diag}

\DeclareMathOperator{\Stab}{Stab}

\DeclareMathOperator{\med}{med}

\DeclareMathOperator{\sinf}{inf}

\newcommand{\fM}{\mathfrak{M}}

\newcommand{\cL}{\mathcal{L}}

\newcommand{\cM}{\mathcal{M}}
\newcommand{\cN}{\mathcal{N}}

\newcommand{\sA}{\mathscr{A}}

\newcommand{\bN}{\mathbb{N}}

\makeatother

\DeclareMathOperator{\nf}{nf}

\DeclareMathOperator{\tdcl}{tdcl}

\begin{document}

\title{Stability and stable groups in continuous logic}

\author{Itaï \textsc{Ben Yaacov}}

\address{Itaï \textsc{Ben Yaacov} \\
  Université Claude Bernard -- Lyon 1 \\
  Institut Camille Jordan, CNRS UMR 5208 \\
  43 boulevard du 11 novembre 1918 \\
  69622 Villeurbanne Cedex \\
  France}

\urladdr{\url{http://math.univ-lyon1.fr/~begnac/}}

\thanks{Author supported by
  ANR chaire d'excellence junior THEMODMET (ANR-06-CEXC-007) and
  by Marie Curie research network ModNet}

\svnInfo $Id: StabGrps.tex 986 2009-09-29 09:42:43Z begnac $
\thanks{\textit{Revision} {\svnInfoRevision} \textit{of} \today}

\keywords{stable theory ; continuous logic ; definable group}
\subjclass[2000]{03C45 ; 03C60 ; 03C90}

\begin{abstract}
  We develop several aspects of local and global stability in
  continuous first order logic.
  In particular, we study type-definable groups and genericity.
\end{abstract}

\maketitle

\section*{Introduction}

Continuous first order logic was introduced by A.\ Usvyatsov and the
author in \cite{BenYaacov-Usvyatsov:CFO}, with the declared purpose of
providing a setting in which classical local stability theory could be
developed for metric structures.
The actual development of stability theory there is fairly limited,
mostly restricted to the definability of $\varphi$-types for a stable
formula $\varphi$, the properties of
$\varphi$-independence, and in case the theory is stable,
properties of independence.
Many fundamental results of classical stability theory, and
specifically those related to stable groups, are missing there, and it
is this gap that the present article proposes to fill.

We assume familiarity with \cite{BenYaacov-Usvyatsov:CFO} and follow
the notation used therein.
Throughout $T$ denotes a continuous theory in a language $\cL$.
We do \emph{not} assume that $T$ is complete, so various constants,
such as $k(\varphi,\varepsilon)$ of \fref{fct:StabPhiTypeDef}, are
uniform across all completions of $T$ (provided that $\varphi$ is
stable in $T$, i.e., in every completion of $T$ separately).

By a \emph{model} we always mean a model of $T$.
Whenever this is convenient, we shall assume that such a model $\cM$
is embedded elementarily in a large monster model $\fM$, i.e., in a
strongly $\kappa$-homogeneous and saturated model, where $\kappa$ is
much bigger than the size of any set of parameters under
consideration.
Notice that we may not simply choose a single monster model for $T$,
as this would consist of choosing one completion.

\section{General reminders}

We shall consider throughout a formula $\varphi(\bar x,\bar y)$ whose
variables are split in two groups.
We recall from \cite{BenYaacov-Usvyatsov:CFO} that a
\emph{definable $\varphi$-predicate}
is a definable predicate $\psi(\bar x$), possibly
with parameters,
which is equivalent to an infinitary continuous combination of
instances of $\varphi$:
\begin{gather*}
  \psi(\bar x)
  \equiv
  \theta\bigl( \varphi(\bar x,\bar b_n) \bigr)_{n\in\bN},
  \qquad \theta\colon [0,1]^\bN \to [0,1] \text{ continuous}.
\end{gather*}
Equivalently, $\varphi(\bar x)$ is a $\varphi$-predicate if it can
be approximated arbitrarily well by finite continuous combinations of
instances of $\varphi$, possibly restricted to the use of the
connectives $\neg$, $\half$, $\dotminus$ alone.

Local types, i.e., $\varphi$-types for a fixed formula $\varphi$,
are discussed in \cite[Section~6]{BenYaacov-Usvyatsov:CFO}.
For a model $\cM$ and a tuple $\bar a$ in some extension
$\cN \succeq \cM$,
the \emph{$\varphi$-type} of $\bar a$ over $\cM$,
denoted $\tp_\varphi(\bar a/M)$, is the partial type given by
$\{
\varphi(\bar x,\bar b) = \varphi(\bar a,\bar b)
\}_{\bar b \in M}$.
The space of all $\varphi$-types over $M$ is denoted $\tS_\varphi(M)$,
and it is a compact Hausdorff quotient of $\tS_n(M)$.
If $\psi(\bar x)$ is a $\varphi$-predicate over $M$
then $\tp_\varphi(\bar a/M)$ determines $\psi(\bar a)$,
so we may identify $\psi$ with a mapping
$\hat \psi\colon \tS_\varphi(M) \to [0,1]$,
sending $p \mapsto \psi^p$.
Every such mapping is continuous, and conversely, every continuous
mapping from $\tS_\varphi(M)$ to $[0,1]$ is of this form.

For $A \subseteq M$ we define $\tS_\varphi(A)$ to be the quotient of
$\tS_\varphi(M)$ where two types are identified if all
$A$-definable $\varphi$-predicates agree on them.
This is again a compact Hausdorff space, a common quotient of
$\tS_\varphi(M)$ and of $\tS_k(A)$ (for the appropriate $k$),
and the continuous mappings $\tS_\varphi(A) \to [0,1]$ are precisely
the $A$-definable $\varphi$-predicates.
In particular this does not depend on the choice of $\cM$.

\begin{lem}
  \label{lem:CompactDef}
  Let $\cM$ be a structure, $K \subseteq M^\ell$ a (metrically) compact
  set and let $\varphi(\bar x,\bar y)$ be a formula
  (or a definable predicate, which we may always name by a new
  predicate symbol without adding any structure).
  Then $\inf_{\bar y\in K} \varphi(\bar x,\bar y)$ is a
  $\varphi$-predicate (with parameters in $K$)
  and for any tuple $\bar x$, the infimum
  is attained by some $\bar y \in K$.

  In particular, $K$ is definable in $\cM$.
\end{lem}
\begin{proof}
  Since $K$ is compact we can find a sequence
  $\{\bar c_n\}_{n\in\bN} \subseteq K$ such that for every
  $\varepsilon > 0$ there is $m = m(\varepsilon)$ such
  that
  $K \subseteq \bigcup_{n<m} B(\bar c_n,\varepsilon)$.
  Then $\sinf_{\bar y\in K} \varphi(\bar x,\bar y)$
  is arbitrarily well approximated by
  formulae of the form
  $\bigwedge_{n<m} \varphi(\bar x,\bar c_n)$ as $m \to \infty$.
  Finally, the infimum of a continuous function on a compact set is
  always attained.
\end{proof}

It will also be convenient to adopt the following somewhat non
standard terminology:
\begin{dfn}
  Let $\cM$ be a model, $A \subseteq M$ a subset.
  We say that $\cM$ is \emph{saturated over $A$}
  if it is strongly $(|A|+\aleph_0)^+$-homogeneous and saturated.
  (In fact, for all intents and purposes it will suffice to require
  $\cM$ to be strongly $\aleph_1$-homogeneous and saturated
  once every member of $A$ is named.)

  We say that a partial type $\pi(\bar x)$ over $\cM$
  is \emph{$A$-invariant} if
  $\cM$ is saturated over $A$
  and $\pi$ is fixed by the action of $\Aut(\cM/A)$.
\end{dfn}

An essential notion for the study of definability of types and
canonical bases in a stable theory is that of imaginary elements and
sorts.
Let us give a brief reminder of their construction, as given in
\cite{BenYaacov-Usvyatsov:CFO}.
Consider a definable predicate with parameters in some set $A$,
let us denote it by $\varphi(\bar x,A)$.
Then we may assume that $A$ is countable, say $A = (a_n)_{n\in\bN}$,
and express $\varphi(\bar x,A)$ as a uniform limit of formulae
$\varphi_n(\bar x,a_{<n})$.
Furthermore, using a forced limit argument, we may assume that the
sequence $\varphi_n(\bar x,y_{<n})$ converges uniformly to some
infinitary definable predicate $\varphi(\bar x,Y)$, giving sense to
$\varphi(\bar x,B)$ for any sequence $B = (b_n)_n$.
If $\cM$ is any structure, we equip $M^\bN$ with the pseudo-metric
$d_\varphi(B,C)
= \sup_{\bar x}\, |\varphi(\bar x,B)-\varphi(\bar x,C)|$
and define the
\emph{sort of canonical parameters for $\varphi$ in $\cM$},
denoted $S_\varphi^\cM$, as the complete metric space associated to
$(M^\bN,d_\varphi)$.
In other words, we divide $M^\bN$ by the kernel
$d_\varphi(Y,Z) = 0$, obtaining a true metric on the quotient, and
pass to the completion.
The predicate $\varphi(\bar x,Y)$ is uniformly continuous with respect
to $Y$ in the metric $d_\varphi$, and therefore passes first to the
quotient and then to the completion, thus inducing a uniformly
continuous predicate
$P_\varphi(\bar x,z)$, where $z \in S_\varphi$.
We may now add $(S_\varphi,d_\varphi)$ as a sort to $\cM$ and equip the new
structure with an additional predicate symbol for $P_\varphi$.
This does not add structure to the original sorts of $\cM$,
elementary embeddings of structures commute with this construction
(by which we mean, in particular, that an elementary embedding
$\cM \preceq \cN$ extends uniquely to
$(\cM,S_\varphi^\cM) \preceq (\cN,S_\varphi^\cN)$),
elementary classes and model completeness thereof are respected by
this construction, and so.
The construction can be slightly simplified
when $\varphi$ only uses finitely many parameters, e.g.,
if it is an honest formula rather than a definable predicate,
but we are going to need the general case.

If $c = [A]$ is the image of $A$ in $S_\varphi$ then $c$ is indeed a
canonical parameter for $\varphi(\bar x,A)$, in the sense that an
automorphism of $\cM$ or of an elementary extension thereof
(and such an automorphism extends uniquely to $S_\varphi$)
fixes $c$ if and only if it fixes the predicate
$\varphi(\bar x,A)$.
By construction we have $P_\varphi(\bar x,c) = \varphi(\bar x,A)$,
and by a convenient abuse of notation we shall permit ourselves to
write $\varphi(\bar x,c)$ instead of either one.

By an \emph{imaginary sort} we mean any sort added in this fashion,
and by \emph{imaginary elements} we mean members of such a sort.
We may repeat this construction for any other definable predicate
$\psi(\bar x',Y')$, or for any family of predicates.
A delicate point here is that even with a countable language one can
construct continuum many definable predicates for whose
canonical parameters imaginary sorts may be added.
For the purposes of stability theory, however, no more imaginary sorts
than the size of the language are truly required: we need the sort
$S_{d\varphi}$ for each formula $\varphi$
(see \fref{fct:StabPhiTypeDef} for the predicate $d\varphi$).
Therefore, with some intentional ambiguity, by $\cM^{eq}$ we shall
usually mean
``$\cM$ along with all the imaginary sorts we are going to need'',
e.g., all the sorts $S_{d\varphi}$.
By $\dcl^{eq}$, $\acl^{eq}$, etc., we mean the respective operations
in the structure $\cM^{eq}$.

\begin{lem}
  \label{lem:InvariantDefinable}
  Let $A$ be a set of parameters and let $\varphi(\bar x)$
  be a definable predicate with parameters possibly outside $A$.
  Then $\varphi$ is $A$-definable if and only if it is $A$-invariant,
  i.e., if and only if its canonical base belongs to $\dcl^{eq}(A)$.

  Similarly, a set which is definable (type-definable) with some
  parameters is definable (type-definable) over $A$ if and only it is
  $A$-invariant.
\end{lem}
\begin{proof}
  First, let us consider an arbitrary surjective continuous mapping
  $\pi\colon X \to Y$ between two compact Hausdorff topological spaces.
  Then $\pi$ is also closed, so $F \subseteq Y$ is closed if and only if
  $\pi^{-1}(F)$ is closed in $X$.
  Since $\pi$ is surjective, $U \subseteq Y$ is open if and only if
  $\pi^{-1}(U)$ is open, and a mapping
  $f\colon Y \to [0,1]$ is continuous if and only if $f \circ \pi$ is
  continuous.

  The assertion now follows from applying the previous paragraph to
  the restriction mapping $\tS_n(B) \to \tS_n(A)$,
  where $B \supseteq A$ contains all the needed parameters, using the
  correspondence between type-definable sets and closed sets, and
  between definable predicates and continuous functions.
  For a definable set $X$, just argue for the definable predicate
  $d(\bar x,X)$.
\end{proof}

\begin{fct}
  \label{fct:TransActLocalTypes}
  \cite[Lemma~6.8]{BenYaacov-Usvyatsov:CFO}
  Let $\varphi(\bar x,\bar y)$ be any formula, $A$ a set,
  $\cM$ a saturated model over $A$, and let $p \in \tS_\varphi(A)$.
  Then $\Aut(\cM/A)$ acts transitively on the set of extensions of $p$
  in $\tS_\varphi(\acl^{eq}(A))$.
\end{fct}

The following notion and fact also appear
(and are used much more extensively)
in \cite[Section~1]{BenYaacov:DefinabilityOfGroups}:
\begin{dfn}
  \label{dfn:LogNeighb}
  Let $X$ and $Y$ be two type-definable sets.
  We say that $Y$ is a \emph{logical neighbourhood} of $X$, in symbols
  $X < Y$, if there is a set of parameters $A$ over which both $X$ and
  $Y$ are defined such that $[X] \subseteq [Y]^\circ$ in $\tS_n(A)$.
\end{dfn}

Notice that the interior of $[Y]$ does depend on $A$
(i.e., if $A' \supseteq A$ then $[Y]^\circ$
calculated in $\tS_n(A')$ may be larger than the
pullback of the interior of $[Y]$ in $\tS_n(A)$).
We may nonetheless choose any parameter set we wish:

\begin{lem}
  \label{lem:LogNeighb}
  Assume that $X$ is type-definable with parameters in $B$, $Y$
  type-definable possibly with additional parameters not in $B$.
  Then:
  \begin{enumerate}
  \item If $X < Y$ then $[X] \subseteq [Y]^\circ$ in $\tS_n(A)$
    for any set $A$ over which both
    $X$ and $Y$ are defined.
  \item If $X < Y$ then there is an intermediate logical neighbourhood
    $X < Z < Y$, which can moreover be taken to be the zero set of a formula
    with parameters in $B$.
  \item If $Y\cap X = \emptyset$ then there is a logical neighbourhood
    $Z > X$ such that $Z \cap Y = \emptyset$.
    Moreover, we may take $Z$ to be a zero set defined over $B$.
  \end{enumerate}
\end{lem}
\begin{proof}
  Assume $X < Y$, where $X$ is type-definable over $B$, and $Y$ over
  $A \supseteq B$.
  Let $\Phi$ consist of all formulae $\varphi(\bar x)$ over $B$ which
  are zero on $X$.
  If $\varphi,\psi \in \Phi$ then $\varphi\vee\psi \in \Phi$,
  and $X$ is defined by the partial type
  $p(\bar x)
  = \{\varphi(\bar x) \leq r\colon \varphi \in \Phi, r > 0\}$.
  By compactness in $\tS_n(A)$
  there is a condition $\varphi(\bar x) \leq r$ in $p(\bar x)$
  which already implies $\bar x \in Y$.
  Let $Z$ be the zero set of
  the formula $\varphi(\bar x) \dotminus r'$
  where $0 < r' = \frac{k}{2^{-m}} < r$.

  Then in $\tS_n(A)$ we have
  $[X] \subseteq [\varphi(\bar x) < r']
  \subseteq [\varphi(\bar x) \leq r']
  \subseteq [\varphi(\bar r) < r]
  \subseteq [Y]$,
  i.e.,
  $[X] \subseteq [Z]^\circ \subseteq [Z] \subseteq [Y]^\circ$,
  proving the first two items.
  The third item now follows from the fact that $\tS_n(A)$ is a normal
  topological space.
\end{proof}

\section{Definability and forking of local types}

Having fixed a theory $T$, we shall call here a formula
$\varphi(\bar x,\bar y)$ \emph{stable} if it is stable in $T$,
that is, if it does not have the order property in any model of $T$.
The order property was defined for continuous logic in
\cite{BenYaacov-Usvyatsov:CFO}, but the reader may simply use
\fref{fct:StabPhiTypeDef} below as the definition of a stable formula.

Let us introduce some convenient notation.
If $\varphi(\bar x,\bar y)$ is any formula with two groups of variables,
$\tilde \varphi(\bar y,\bar x)$ denotes the same formula with the groups of
variables interchanged.
More generally, let us define
\begin{gather*}
  \tilde \varphi^n(\bar y,\bar x_{\leq 2n})
  = \med_n\bigl( \varphi(\bar x_i,\bar y) \bigr)_{i \leq 2n},
\end{gather*}
where $\med_n\colon[0,1]^{2n+1} \to [0,1]$ is the median value
combination:
\begin{gather*}
  \med_n(t_{\leq 2n})
  = \bigwedge_{w \in [2n+1]^{n+1}} \, \bigvee_{i \in w} \, t_i
  = \bigvee_{w \in [2n+1]^{n+1}} \, \bigwedge_{i \in w} \, t_i.
\end{gather*}
Thus in particular $\tilde \varphi^0 = \tilde \varphi$ and every
instance of $\tilde \varphi^n$ is a $\tilde \varphi$-predicate.

\begin{fct}
  \label{fct:StabPhiTypeDef}
  Let $\varphi(\bar x,\bar y)$ be a stable formula.
  Let $\cM$ be a model and let $p \in \tS_\varphi(\cM)$ be a complete
  $\varphi$-type.
  Then
  \begin{enumerate}
  \item The type $p$ is definable over $M$, i.e.,
    there exists an $M$-definable
    $\tilde \varphi$-predicate $d_p\varphi(\bar y)$
    such that
    $\varphi(x,\bar b)^p = d_p\varphi(\bar b)$
    for all $\bar b \in M$.
    Moreover, this definition is uniform, in the sense that
    there exists an infinitary definable predicate
    $d\varphi(\bar y,X)$ which only depends on $\varphi$
    such that $d_p\varphi(\bar y)$
    is equal to an instance $d\varphi(\bar y,C)$ where
    $C \subseteq M$.
  \item For every $\varepsilon > 0$ there exists a number
    $k = k(\varphi,\varepsilon) \in \bN$
    (which depends on $\varphi$ and on $\varepsilon$ but not on $p$)
    and a tuple
    $\bar c_{\leq 2k}
    = \bar c^\varepsilon_{\leq 2k(\varphi,\varepsilon)}$
    in $\cM$
    (which does depend on $p$)
    such that
    \begin{gather*}
      |d_p\varphi(\bar y)-\tilde \varphi^k(\bar y,\bar c_{\leq 2k})|
      < \varepsilon.
    \end{gather*}
  \item
    Assume moreover that $\cM$ is saturated over some subset
    $A \subseteq M$.
    Then in the previous item the tuples
    $\bar c_{\leq 2k}$ can be chosen so that
    each $\bar c_n$ realises $p \rest_{A\bar c_{<n}}$.
  \end{enumerate}
\end{fct}
\begin{proof}
  The first two items are taken from
  \cite[Lemma~7.4 and Proposition~7.6]{BenYaacov-Usvyatsov:CFO}.
  The third item, while not explicitly stated there,
  is immediate from the proof.
\end{proof}

By abuse of notation we may sometimes write
$d_p\varphi(\bar y) = d\varphi(\bar y,c)$, where $c \in M^{eq}$ is
the canonical parameter for the definition.
This canonical parameter is called the \emph{canonical base} of $p$,
denoted $\Cb(p)$.


We recall that for $A \subseteq B \subseteq \cM$,
$p \in \tS_\varphi(B)$ \emph{does not fork} over $A$
if it admits an extension $p_1 \in \tS_\varphi(M)$
which is definable over $\acl^{eq}(A)$.
In this case $p_1$ itself does not fork over $A$ or $B$.
A type over a model clearly admits a unique non forking extension to
any larger model (and therefore set), so this definition does not
depend on the choice of $\cM$.

We proved in \cite[Proposition~7.15]{BenYaacov-Usvyatsov:CFO}
that every $\varphi$-type over a set $A$ admits a non forking extension to
every model (and therefore every set) containing $A$.
A minor enhancement of that result will be quite useful.

\begin{lem}[Existence of non forking extensions]
  \label{lem:NFExtExist}
  Let $\varphi(\bar x,\bar y)$ be a stable formula, $A$ a set,
  $\cM \supseteq A$ a saturated model over $A$.
  Let $\pi(\bar x)$ be a consistent $A$-invariant partial type over $M$.
  Then there exists $p \in \tS_\varphi(M)$ compatible with $\pi$
  which does not fork over $A$.
\end{lem}
\begin{proof}
  Let
  $X = \{p \in \tS_\varphi(M)\colon p\cup \pi \text{ is consistent}\}$.
  Then $X$ is non empty and $A$-invariant.
  By \cite[Lemma~7.14]{BenYaacov-Usvyatsov:CFO},
  there is $Y \subseteq X$ which is $A$-good,
  i.e., which is $A$-invariant and metrically compact.
  By \cite[Lemma~7.13]{BenYaacov-Usvyatsov:CFO}, any $p \in Y$ would do.
\end{proof}

\begin{cor}
  \label{cor:NonForkingInvariant}
  Let $\varphi(\bar x,\bar y)$ be a stable formula, $A$ a set,
  $\cM \supseteq A$ a saturated model over $A$.
  Then $p \in \tS_\varphi(M)$ does not fork over $A$ if and only if it
  is $\acl^{eq}(A)$-invariant.
\end{cor}
\begin{proof}
  Left to right follows from the definition, right to left from
  \fref{lem:NFExtExist}.
\end{proof}

\begin{cor}
  \label{cor:CompleteNFExt}
  Let $A$ be a set, $\cM \supseteq A$ a saturated model over $A$ and
  $\pi(\bar x)$ a consistent $A$-invariant partial type over $M$.
  Then there exists a complete type $p$ such that
  $\pi \subseteq p \in \tS_n(M)$,
  and for every stable formula $\varphi(\bar x,\bar y)$ the
  restriction $p\rest_\varphi \in \tS_\varphi(M)$ does not fork over
  $A$.
\end{cor}
\begin{proof}
  We may assume that $A = \acl^{eq}(A)$.
  Index all stable formulae of the form
  $\varphi_i(\bar x,\bar y_i)$ by $i < \lambda$.
  We define an increasing sequence of consistent $A$-invariant partial
  types $\pi_i$ over $M$, starting with $\pi_0 = \pi$.
  Given $\pi_i$, by \fref{lem:NFExtExist} there is
  $p_i \in \tS_{\varphi_i}(M)$ be non forking over $A$ and compatible
  with $\pi_i$, so $\pi_{i+1} = \pi_i \cup p_i$ is consistent and
  $A$-invariant.
  For limit $i$ we define $\pi_i = \bigcup_{j<i} \pi_j$.
  Finally, let $p \in \tS_n(M)$ be any completion of $\pi_\lambda$.
  Then $p$ will do.
\end{proof}

It follows that if the theory is stable then every complete type over
a set admits non forking extensions.
The same fact was proved in \cite{BenYaacov-Usvyatsov:CFO} using a
somewhat longer ``gluing'' argument.

\begin{fct}[Symmetry
  {\cite[Proposition~7.16]{BenYaacov-Usvyatsov:CFO}}]
  \label{fct:NFSymmetry}
  Let $\cM$ be a model,
  $p(\bar x) \in \tS_\varphi(M)$,
  $q(\bar y) \in \tS_{\tilde \varphi}(M)$.
  Then $d_p\varphi(\bar y)^q = d_q\tilde \varphi(\bar x)^p$.
\end{fct}

\begin{prp}
  \label{prp:NFExtDef}
  Let $\varphi(\bar x,\bar y)$ be a stable formula,
  $\cM$ a model, $A \subseteq M$.
  For each $\bar b \in M$ let
  $\chi_{\bar b}(\bar x)$ be the definition of a non forking
  extension of
  $\tp_{\tilde \varphi}(\bar b/\acl^{eq}(A))$ to $M$.
  \begin{enumerate}
  \item Each $\chi_{\bar b}(\bar x)$ is a definable
    $\varphi$-predicate over $\acl^{eq}(A)$.
  \item
    A $\varphi$-type $p \in \tS_\varphi(M)$ does not fork over $A$
    if and only if $\varphi(\bar x,\bar b)^p = \chi_{\bar b}(\bar x)^p$
    for all $\bar b \in M$.
  \item \label{item:NFExtDefStationary}
    A $\varphi$-type over $\acl^{eq}(A)$ is \emph{stationary}, i.e.,
    admits a unique non forking extension to every larger set.
  \item Let
    $r(\bar x) = \big\{
    |\varphi(\bar x,\bar b)-\chi_{\bar b}(\bar x)| = 0
    \big\}_{\bar b \in M}$.
    Then the partial type $r(\bar x)$ defines the set of $\varphi$-types which
    do not fork over $A$:
    \begin{gather*}
      \bar a \models r
      \quad \Longleftrightarrow \quad
      \tp_\varphi(\bar a/M) \text{ does not fork over } A.
    \end{gather*}
  \item For every $B \supseteq A$, the set
    $\{p \in \tS_\varphi(B)\colon p \text{ does not fork over } A\}$ is closed.
  \end{enumerate}
\end{prp}
\begin{proof}
  The first item is by \fref{fct:StabPhiTypeDef} and the definition of
  non forking.

  For the second, fix $\bar b \in M$, let
  $q_0 = \tp_{\tilde \varphi}(\bar b/\acl^{eq}(A))$ and let
  $q \in \tS_\varphi(M)$ be the non forking extension defined by
  $\chi_{\bar b}$.
  Assume $p \in \tS_\varphi(M)$ does not fork over $M$, so
  $d_p\varphi(\bar y)$ is a $\tilde \varphi$-predicate over
  $\acl^{eq}(A)$.
  By \fref{fct:NFSymmetry},
  \begin{gather*}
    \varphi(\bar x,\bar b)^p
    = d_p\varphi(\bar b)
    = d_p\varphi(\bar y)^{q_0}
    = d_p\varphi(\bar y)^q
    = d_q\tilde \varphi(\bar x)^p
    = \chi_{\bar b}(\bar x)^p.
  \end{gather*}
  Conversely, assume that
  $\varphi(\bar x,\bar b)^p
  = \chi_{\bar b}(\bar x)^p$ for all $\bar b \in M$,
  and let $p' \in \tS_\varphi(M)$ be any non forking extension
  of $p\rest_{\acl^{eq}(A)}$.
  Then $p = p'$, proving also the third item.
  The fourth item is just a re-statement of the second.

  For the last item we may assume that $B \subseteq M$.
  The set $[r] \subseteq \tS_\varphi(M)$ is closed, and so is its projection to
  $\tS_\varphi(B)$.
  This projection is precisely the
  set of types which do not fork over $A$.
\end{proof}

\fref{prp:NFExtDef}.\fref{item:NFExtDefStationary} is
the analogue of the finite equivalence relation theorem in continuous
logic.
It has already appeared as
\cite[Proposition~7.17]{BenYaacov-Usvyatsov:CFO}.
In case $p \in \tS_\varphi(A)$ is stationary,
the unique non forking extension to $B \supseteq A$
will be denoted $p\rest^B$.
Similarly, we write $d_p\varphi$ for the definition of $p\rest^M$
where $\cM \supseteq A$ is any model (and this does not depend on the
choice of $\cM$).
Thus, in hindsight, in the statement of \fref{prp:NFExtDef},
the definitions $\chi_{\bar b}$ are uniquely determined,
$\chi_{\bar b} = d_{\bar b/\acl^{eq}(A)}\tilde \varphi$.

\begin{cor}
  \label{cor:TransActNonForking}
  Let $\varphi(\bar x,\bar y)$ be a stable formula, $A$ a set,
  $\cM$ a saturated model over $A$.
  Let $p \in \tS_\varphi(A)$.
  Then $\Aut(\cM/A)$ acts transitively on
  the set of non forking extensions of $p$ in $\tS_\varphi(M)$.
  If $T$ is stable and $p \in \tS_n(A)$
  then $\Aut(\cM/A)$ acts transitively on the set of non forking
  extensions of $p$ to $\cM$.
\end{cor}
\begin{proof}
  The first assertion follows from \fref{fct:TransActLocalTypes} and
  \fref{prp:NFExtDef}.\fref{item:NFExtDefStationary}.
  For the second we need the even easier fact that $\Aut(\cM/A)$ acts
  transitively on the extensions of a complete type $p \in \tS_n(A)$
  to $\acl^{eq}(A)$.
\end{proof}

\begin{cor}
  \label{cor:DefinableStationary}
  Let $\varphi(\bar x,\bar y)$ be a stable formula $\cM$ a model.
  A type $p \in \tS_\varphi(M)$
  is definable over $A$ if and only if
  it does not fork over $A$ and $p\rest_A$
  is stationary.
\end{cor}
\begin{proof}
  We may assume that $\cM$ saturated over $A$.
  Let $p' \in \tS_\varphi(M)$ be any non forking extension of
  $p\rest_A$.
  By \fref{cor:TransActNonForking} there is an automorphism
  $f \in \Aut(\cM/A)$ sending $p$ to $p'$.
  If $p$ is definable over $A$ then $p' = f(p) = p$.
  Conversely, if $p$ does not fork over $A$ and
  $p\rest_A$ is stationary then
  $\Aut(\cM/A)$ fixes $p$ and therefore fixes $d_p\varphi$.
  By \fref{lem:InvariantDefinable},
  the latter is over $A$.
\end{proof}

\begin{cor}
  \label{cor:CompatNFExt}
  Let $\varphi(\bar x,\bar y)$ be a stable formula,
  $A$ a set, $q(\bar x) \in \tS_n(A)$ a complete type
  over $A$, and let $p_0 = q\rest_\varphi \in \tS_\varphi(A)$.
  Then $q$ is compatible with every non forking extension of $p_0$.
\end{cor}
\begin{proof}
  By \fref{lem:NFExtExist}, $q$ is compatible with at least one
  non forking extension of $p$ to the monster model.
  By \fref{cor:TransActNonForking} it is compatible with all of them.
\end{proof}

We pass to forking of single conditions.
\begin{dfn}
  Let $\varphi(\bar x,\bar b)$ be an instance of a stable formula,
  $A$ a set.
  We say that a condition $\varphi(\bar x,\bar b) \leq r$
  \emph{does not fork} over $A$
  if there exists a $\varphi$-type $p \in \tS_\varphi(A\bar b)$
  non forking over $A$ such that
  $\varphi(\bar x,\bar b)^p \leq r$.
\end{dfn}

\begin{prp}
  \label{prp:NFCondition}
  Let $\varphi(\bar x,\bar b)$ be an instance of a stable formula,
  $A$ a set of parameters.
  Then the following are equivalent:
  \begin{enumerate}
  \item The condition $\varphi(\bar x,\bar b) \leq r$ does not fork
    over $A$.
  \item
    Every family of
    $\acl^{eq}(A)$-conjugates of $\varphi(\bar x,\bar b) \leq r$
    is consistent.
  \item For every set $B \supseteq A,\bar b$ there exists a complete
    type $p \in \tS_n(B)$ such that
    $p\rest_\psi$ does not fork over $A$ for any stable formula $\psi$
    (if $T$ is stable: $p$ does not fork over $A$)
    and $\varphi(\bar x,\bar b)^p \leq r$.
  \end{enumerate}
\end{prp}
\begin{proof}
  \begin{cycprf}
  \item[\impnext]
    Let $p$ witness that $\varphi(\bar x,\bar b) \leq r$ does not fork
    over $A$.
    Then any non forking extension of $p$ to a large
    model is $\acl^{eq}(A)$-invariant.
  \item[\impnext]
    We may assume that $B = \cM$ is saturated over $A$.
    Let $\pi$ consist of all the $\acl^{eq}(A)$-conjugates of
    $\varphi(\bar x,\bar b) \leq r$ in $\cM$.
    It is consistent by assumption and $\acl^{eq}(A)$-invariant by
    construction so we may apply \fref{cor:CompleteNFExt}.
  \item[\impfirst]
    Immediate.
  \end{cycprf}
\end{proof}

We may define the \emph{non forking degree} of
$\varphi(\bar x,\bar b)$ over $A$ to be
\begin{gather*}
  \nf\bigl( \varphi(\bar x,\bar b)/A \bigr)
  = \inf\bigl\{
  r\colon \varphi(\bar x,\bar b) \leq r \text{ does not fork over } A
  \bigr\}.
\end{gather*}
An easy compactness argument shows that
the infimum is attained and the condition
$\varphi(\bar x,\bar b)
\leq \nf\bigl( \varphi(\bar x,\bar b)/A \bigr)$
does not fork over $A$.
In addition, by the existence of non-forking types we have
$\nf\bigl( \varphi(\bar x,\bar b)/A \bigr)
+ \nf\bigl( \neg\varphi(\bar x,\bar b)/A \bigr) \leq 1$.

\begin{dfn}
  A \emph{faithful continuous connective} in $\alpha$ variables is a
  continuous function $\theta\colon [0,1]^\alpha \to [0,1]$ satisfying
  $\inf \bar a \leq \theta(\bar a) \leq \sup \bar a$.

  If $\theta\colon [0,1]^\alpha \to [0,1]$
  is a faithful continuous connective and
  $(\varphi_i)_{i<\alpha}$ a sequence of definable predicates,
  then the definable predicate $\theta(\varphi_i)_{i < \alpha}$
  is called a
  \emph{faithful combination} of $(\varphi_i)_{i < \alpha}$.
\end{dfn}

Since a continuous function to $[0,1]$ can only take into account
countably many arguments, we may always assume that $\alpha \leq \omega$.
Notice that any connective constructed using $\vee$ and $\wedge$ alone
is faithful
(so in particular the median value connective
$\med_n\colon [0,1]^{2n+1} \to [0,1]$ is).
Similarly, any uniform limit of faithful combinations is faithful.

\begin{lem}
  Let $\varphi(\bar x,\bar y)$ be a stable formula.
  Let $A = \acl^{eq}(A)$ be a set of parameters, $\bar a$ a tuple,
  $|\bar x| = |\bar a|$.
  Let $p = \tp_\varphi(\bar a/A)$.
  Then $d_p\varphi(\bar x,\bar y)$ is a faithful combination of
  $A$-conjugates of $\varphi(\bar a,\bar y)$.
\end{lem}
\begin{proof}
  By the preceding discussion and the last item of
  \fref{fct:StabPhiTypeDef}.
\end{proof}

\begin{lem}
  \label{lem:NonStationaryDefinition}
  Let $\varphi(\bar x,\bar b)$ be an instance of a stable formula,
  $A$ a set of parameters.
  Then there exists an $A$-definable predicate $\psi(\bar x)$
  such that for every tuple $\bar a$ (not necessarily in $A$):
  \begin{align*}
    \psi(\bar a) &
    =
    \inf \{ \varphi(\bar x,\bar b)^p \colon
    p \in \tS_\varphi(A\bar b)
    \text{ is a non forking extension of } \tp_\varphi(\bar a/A)\}
    \\ &
    =
    \inf \{ \varphi(\bar a,\bar y)^q \colon
    q \in \tS_\varphi(A\bar a)
    \text{ is a non forking extension of } \tp_{\tilde \varphi}(\bar b/A)\}.
  \end{align*}
  Moreover, $\psi(\bar x)$ can be taken to be a faithful combination
  of $A$-conjugates of $\varphi(\bar x,\bar b)$.
\end{lem}
\begin{proof}
  Fix a model $\cM \supseteq A,\bar b$, saturated over $A$.
  Let $G = \Aut(\cM/A)$.
  Let $q \in \tS_{\tilde \varphi}(M)$
  be the unique non forking extension of
  $\tp_{\tilde \varphi}(\bar b/\acl^{eq}(A))$.
  Let $\chi(\bar x,c) = d_q\tilde \varphi(\bar x)$,
  where $c \in \acl^{eq}(A)$ is the canonical parameter for the
  definition.
  By the previous Lemma, $\chi(\bar x,c)$ is a faithful combination of
  $\acl^{eq}(A)$-conjugates of $\varphi(\bar x,\bar b)$.

  Let $C$ be the set of $A$-conjugates of $c$.
  Since $c$ is algebraic over $A$, $C$ is (metrically)
  compact.
  By \fref{lem:CompactDef} 
  $\psi(\bar x) = \inf_{c' \in C} \chi(\bar x,c')$
  is a continuous combination
  of instances $\chi(\bar x,c')$ with $c' \in C$, i.e., of
  $A$-conjugates of $\chi(\bar x,c)$, and it is clearly a faithful
  combination.
  Thus $\psi(\bar x)$ is a faithful combination of $A$-conjugates of
  $\varphi(\bar x,\bar b)$, and it is clearly over $A$.

  We may assume that $\bar a \in M$, and let
  $p \in \tS_\varphi(M)$
  be the unique non forking extension of
  $\tp_\varphi(\bar a/\acl^{eq}(A))$.
  Then
  \begin{align*}
    \psi(\bar a) &
    = \inf_{g \in G} \chi(\bar a,gc)
    = \inf_{g \in G} d_{gq}\tilde \varphi(\bar a)
    = \ldots
    \\ &
    \ldots
    = \inf_{g \in G} d_{g^{-1}p}\varphi(\bar y)^q
    = \inf_{g \in G} \varphi(\bar x,\bar b)^{gp},
    \\ &
    \ldots
    = \inf_{g \in G} \varphi(\bar a,\bar y)^{gq}.
  \end{align*}
  Since
  $\{gp\}_{g\in G}$
  and
  $\{gq\}_{g\in G}$
  are the sets of non forking extensions of
  $\tp_\varphi(\bar a/A)$ and of $\tp_{\tilde \varphi}(\bar b/A)$, respectively, to
  $M$, we are done.
\end{proof}

\begin{thm}[Open Mapping Theorem]
  Assume $T$ is stable, and let $A \subseteq B$ be any sets of
  parameters.
  Let $X \subseteq \tS_n(B)$ be the set of types which do not fork
  over $A$.
  Then $X$ is compact and the restriction mapping
  $\rho_A\colon X \to \tS_n(A)$
  sending $p \mapsto p\rest_A$ is an open continuous surjective
  mapping.
\end{thm}
\begin{proof}
  We already know that $X$ is compact and that $\rho_A$ is
  continuous and surjective.

  Consider a basic open subset $U \subseteq X$, of the form
  $U = X \cap [\varphi(\bar x,\bar b) < 1]$.
  Let $\psi(\bar x)$ be as in \fref{lem:NonStationaryDefinition} and
  let
  $V = [\psi(\bar x) < 1] \subseteq \tS_n(A)$.
  By \fref{cor:CompleteNFExt} every $\varphi$-type over $B$ which does
  not fork over $A$ extends to a complete type over $B$ which does not
  fork over $A$, whence $V = \rho_A(U)$.
\end{proof}

Notice that a similar proof yields that if $\varphi(\bar x,\bar y)$
is stable then the restriction mapping
$\rho_{A,\varphi}\colon X_\varphi \to \tS_\varphi(A)$ is open,
where $X_\varphi \subseteq \tS_\varphi(B)$ denotes the set of
$\varphi$-types which do not fork over $A$.

It follows from \fref{lem:NonStationaryDefinition}
that a $\tilde \varphi$-type (and therefore a $\varphi$-type)
over an arbitrary set $A$ is definable over $A$, but of course the
same definition applied to a larger set need not give a consistent
complete type.
This yields the following (adaptation of a) classical result:

\begin{thm}[Separation of variables]
  Let $\varphi(\bar x,\bar b)$ be an instance of a stable formula,
  and let $X$ be a type-definable set in the sort of $\bar x$,
  say with parameters in $A$.
  Then there is a subset (at most countable) $B \subseteq X$
  and a $B$-definable predicate $\psi(\bar x)$ such that
  $\psi(\bar x)\rest_X = \varphi(\bar x,\bar b)\rest_X$.

  Moreover, $\psi(\bar x)$ can be taken to be a faithful combination
  of instances $\varphi(\bar x,\bar b')$
  such that $\bar b' \equiv_B \bar b$
  (or even $\bar b' \equiv_{B'} \bar b$ where $B' \subseteq X$ is an
  arbitrary small subset).
\end{thm}
\begin{proof}
  Fix a model $\cM \supseteq A,\bar b$, saturated over $A$,
  and let $C = X(\cM)$.
  Let $\psi(\bar x)$ be as in \fref{lem:NonStationaryDefinition}.
  Then $\psi(\bar x)$ is definable over $C$ and therefore over $B$
  where $B \subseteq C$ is an appropriate countable subset.
  Then for all $\bar a \in C$ we have
  $\psi(\bar a)
  = \varphi(\bar a,\bar y)^{\tp_{\tilde \varphi}(\bar b/C)}
  = \varphi(\bar a,\bar b)$.
  Now let $\fM$ be the monster model and
  $\bar a \in X = X(\fM)$.
  By saturation of $\cM$ we can find there some
  $\bar a' \equiv_{AB\bar b} \bar a$.
  Then $\bar a' \in C$
  and
  $\varphi(\bar a,\bar b)
  = \varphi(\bar a',\bar b)
  = \psi(\bar a')
  = \psi(\bar a)$,
  as desired.

  The moreover part follows from the proof.
\end{proof}

It follows that if $X$ is an $A$-type-definable set and
$Y \subseteq X$ is a type-definable subset, then $Y$ is type-definable
over $AB$ for some countable $B \subseteq X$.
If $Y$ is a definable set then it is definable over $AB$
(by \fref{lem:InvariantDefinable}, since the definable predicate
$d(\bar x,Y)$ is $AB$-invariant). 

\begin{prp}
  \label{prp:FaithfulNF}
  Let $\varphi(\bar x,\bar b)$ be an instance of a stable formula,
  $A$ a set of parameters.
  Then the following are equivalent:
  \begin{enumerate}
  \item The condition $\varphi(\bar x,\bar b) \leq r$ does not fork
    over $A$.
  \item There is an $A$-definable predicate $\psi(\bar x)$ which is
    a faithful combination of $A$-conjugates
    of $\varphi(\bar x,\bar b)$
    such that $\psi(\bar x) \leq r$ is consistent.
  \end{enumerate}
\end{prp}
\begin{proof}
  Fix a model $\cM \supseteq A,\bar b$ saturated over $A$.
  \begin{cycprf}
  \item[\impnext]
    Let $\psi(\bar x)$ be as in \fref{lem:NonStationaryDefinition}.
    Let also $p \in \tS_\varphi(A\bar b)$ be non forking over $A$
    such that $\varphi(\bar x,\bar b)^p \leq r$.
    Then
    $\psi(\bar x)^p \leq \varphi(\bar x,\bar b)^p \leq r$.
  \item[\impfirst]
    Let
    $\psi(\bar x) = \theta\bigl(
    \varphi(\bar x,\bar b_n)
    \bigr)_{n\in\bN}$
    be definable over $A$ as in the assumption
    (so $\bar b_n \equiv_A \bar b$ and $\theta$ is a faithful
    continuous connective).

    By \fref{lem:NFExtExist} there exists $p \in \tS_\varphi(M)$
    compatible with $\psi(\bar x) \leq r$ and non forking over $A$,
    so in particular $\acl^{eq}(A)$-invariant.
    Then $\inf_n \varphi(\bar x,\bar b_n)^p \leq r$ by faithfulness,
    so for all $r' > r$ there exists $n$ such that
    $\varphi(\bar x,\bar b_n)^p < r'$.
    Up to an automorphism fixing $A$ we may assume that
    $\varphi(\bar x,\bar b)^p < r'$,
    and by invariance $\varphi(\bar x,\bar b')^p < r'$
    for every $\bar b' \equiv_{\acl^{eq}(A)} \bar b$.

    We have thus shown that for every $r' > r$,
    any set of $\acl^{eq}(A)$-conjugates of
    $\varphi(\bar x,\bar b) \leq r'$ is consistent.
    By compactness the same holds for
    $\varphi(\bar x,\bar b) \leq r$.
  \end{cycprf}
\end{proof}

\section{Heirs and co-heirs}

We turn to study co-heirs, and more generally, approximately realised
partial types, in continuous logic.
In the context of stability, approximate realisability serves as a
criterion for non forking.
For an earlier treatment of co-heirs in the context of metric
structures see \cite[Section~3.2]{BenYaacov:Morley}.

\begin{dfn}
  Let $A \subseteq B$ be two sets of parameters.
  We say that a partial type $\pi$ over $B$
  is \emph{approximately realised} in $A$ if every
  logical neighbourhood (\fref{dfn:LogNeighb}) of $\pi$ over $B$
  is realised in $A$.

  If $\cM$ is a model, $B \supseteq M$, and
  $p \in \tS_n(B)$ is approximately realised in $M$,
  we may say that $p$ is
  a \emph{co-heir} of its restriction to $\cM$.
\end{dfn}

\begin{rmk}
  \begin{enumerate}
  \item The classical logic analogue of an approximately realised type
    is a \emph{finitely realised} one, but this terminology would be
    misleading in the continuous setting.
  \item A complete type over a model $\cM$ is always approximately
    realised there.
    (This is essentially the Tarski-Vaught Criterion.)
  \end{enumerate}
\end{rmk}

\begin{fct}
  Let $A \subseteq B$ and let $\pi(\bar x)$ be a partial type over
  $B$.
  \begin{enumerate}
  \item
    Let $X \subseteq \tS_n(B)$ consist of all
    types over $B$ which are realised in $A$,
    $[\pi] \subseteq \tS_n(B)$ the closed set defined by $\pi$.
    Then $\pi$ is approximately realised in $A$ if and only if
    $[\pi] \cap \overline X \neq \emptyset$.
    In particular, $\overline X$ is the set of all complete
    $n$-types over $B$ which are approximately realised in $A$.
  \item
    If $C \supseteq B$ then $\pi$ is approximately realised in $A$ as
    a partial type over $B$ if and only if it is approximately
    realised in $A$ as a partial type over $C$.
  \item
    If $\pi$ is approximately realised in $A$ then it extends to a
    complete type $\pi \subseteq p \in \tS_n(B)$
    which is approximately realised in $A$.
  \item
    A type over a model $\cM$ admits extensions to arbitrary sets
    which are approximately realised in $M$.
  \end{enumerate}
\end{fct}
\begin{proof}
  We prove the first two items together.
  Clearly if $\pi$ is approximately realised in $A$ as a partial type
  over $C$ then it is approximately realised in $A$ as a partial type
  over $B$, in which case every neighbourhood of $[\pi]$ in $\tS_n(B)$
  intersects $X$ and by a compactness argument $[\pi]$ intersects
  $\overline X$.
  Finally, assume $[\pi] \cap \overline X \neq \emptyset$ and assume
  that $\pi \vdash \varphi(\bar x) > 0$.
  Let $Y = [\varphi = 0] \subseteq \tS_n(C)$ and let $Z$ be its
  projection to $\tS_n(B)$.
  Then $Z$ is compact, $Z \cap [\pi] = \emptyset$, so
  $U = \tS_n(B) \setminus Z$ is a neighbourhood of $[\pi]$.
  By assumption there exists $\bar a \in A$
  such that $\tp(\bar a/B) \in X \cap U$.
  Then
  $\tp(\bar a/C) \notin Y$, i.e., $\varphi(\bar a) > 0$,
  as desired.

  For the third item, any $p \in [\pi] \cap \overline X$ will do.
  For the fourth, use the fact that a type over a model is approximately
  realised there.
\end{proof}

\begin{fct}
  Let $\cN$ be a model saturated over $A \subseteq N$.
  If $p \in \tS_n(N)$ or $p \in \tS_\varphi(N)$
  is approximately realised in $A$ then it is
  $A$-invariant.
\end{fct}
\begin{proof}
  We only consider the case $p \in \tS_\varphi(N)$,
  since the case $p \in \tS_n(N)$ follows from it.
  Say $\bar b,\bar c \in N$, $\bar b \equiv_A \bar c$, and let
  $\varepsilon > 0$ be given.
  By assumption there is $\bar a \in A$ such that
  \begin{gather*}
    |\varphi(\bar a,\bar b) - \varphi(\bar x,\bar b)^p| < \varepsilon/2, \qquad
    |\varphi(\bar a,\bar c) - \varphi(\bar x,\bar c)^p| < \varepsilon/2.
  \end{gather*}
  As we assumed that $\bar b \equiv_A \bar c$ we have in particular
  $\varphi(\bar a,\bar b) = \varphi(\bar a,\bar c)$ and thus
  $|\varphi(\bar x,\bar b)^p - \varphi(\bar x,\bar c)^p|
  < \varepsilon$,
  for every $\varepsilon > 0$.
  We conclude that
  $\varphi(\bar x,\bar b)^p = \varphi(\bar x,\bar c)^p$,
  as desired.
\end{proof}

\begin{lem}
  \label{lem:CoHeirNF}
  Let $A \subseteq B$,
  $p(\bar x) \in \tS_n(B)$ approximately realised in $A$,
  and assume $\varphi(\bar x,\bar y)$ is stable.
  Then $p\rest_\varphi \in \tS_\varphi(B)$ does not fork over $A$.
\end{lem}
\begin{proof}
  Let $\cN \supseteq B$ be saturated over $A$ and let
  $q \in \tS_n(N)$ extend $p$, still approximately realised in $A$.
  Then $q$, and thus $q\rest_\varphi$, are $A$-invariant,
  so $q\rest_\varphi$ does not fork over $A$ and neither does
  $p\rest_\varphi$.
\end{proof}

\begin{prp}
  \label{prp:CoHeirNF}
  Let $\varphi(\bar x,\bar y)$ be a stable formula,
  $\cM$ a model, $A \supseteq M$.
  Let also
  $p(\bar x) \in \tS_\varphi(A)$ be a complete $\varphi$-type, and
  $q(\bar x) \in \tS_n(M)$ a complete  type over $M$ such that
  $p\rest_M = q\rest_\varphi \in \tS_\varphi(M)$.
  Then the following are equivalent:
  \begin{enumerate}
  \item $p \cup q$ is approximately realised in $M$.
  \item $p$ is approximately realised in $M$.
  \item $p$ does not fork over $M$.
  \end{enumerate}
\end{prp}
\begin{proof}
  \begin{cycprf}
  \item[\impnext] Immediate.
  \item[\impnext]
    Find $p'(\bar x) \in \tS_n(A)$ extending $p$ which is
    approximately realised in $M$ and use
    \fref{lem:CoHeirNF}.
  \item[\impfirst]
    Find $q'(\bar x) \in \tS_n(A)$ extending $q$ which is
    approximately realised in $M$.
    Then $q'\rest_\varphi$ is non forking over $M$ by \fref{lem:CoHeirNF},
    so it must be the unique non forking extension of
    $p\rest_M = q\rest_\varphi$.
    Therefore $q\cup p \subseteq q'$ is approximately realised in $M$.
  \end{cycprf}
\end{proof}

Similarly,
\begin{prp}
  Assume $T$ is stable.
  Let $\cM$ be a model of $T$, $A \supseteq M$,
  $p(\bar x) \in \tS_n(A)$.
  Then the following are equivalent:
  \begin{enumerate}
  \item $p$ does not fork over $M$.
  \item $p$ is approximately realised in $M$.
    \setcounter{dummycnt}{\value{enumi}}
  \end{enumerate}
  If $A = \cN \succeq \cM$ is saturated over $M$ then these are
  further equivalent to
  \begin{enumerate}
    \setcounter{enumi}{\value{dummycnt}}
  \item $p$ is $M$-invariant.
  \end{enumerate}
\end{prp}

\begin{dfn}
  Let $\cM$ be a model, $\cM \subseteq B$.
  A type $p \in \tS_n(B)$ is said to be an \emph{heir}
  of its restriction to $M$ if for every formula
  $\varphi(\bar x,\bar b,\bar m)$ with $\bar b \in B$
  and $\bar m \in M$,
  and for every $\varepsilon > 0$,
  there are $\bar b' \in M$ such that
  $|\varphi(\bar x,\bar b,\bar m) - \varphi(\bar x,\bar b',\bar m)|^p
  < \varepsilon$.
\end{dfn}

Clearly every type over a model is an heir of itself.
Also, it is not difficult to check that
if $\cM$ is a model and $\bar a$, $\bar b$ are two tuples possibly
outside $\cM$ then
\begin{gather*}
  \tp(\bar a/M\bar b) \text{ is an heir of } \tp(\bar a/M)
  \quad \Longleftrightarrow \quad
  \tp(\bar b/M\bar a) \text{ is a co-heir of } \tp(\bar b/M).
\end{gather*}
Finally, a standard compactness argument yields that if
$\cM \subseteq B \subseteq C$ and $p \in \tS_n(B)$ is an heir of
$p\rest_M$ then it admits an extension to $C$ which is an heir as
well.

\begin{lem}
  \label{lem:UniqueHeir}
  Let $\cM$ be a model, $p(\bar x) \in \tS_n(M)$.
  Then $p$ is definable if and only if it has a unique heir to every
  superset $B \supseteq \cM$.
\end{lem}
\begin{proof}
  (We follow Poizat \cite[Théorème~11.07]{Poizat:Cours}.)
  For left to right, assume $p$ is definable and let
  $q \in \tS_n(B)$ be an heir of $p$, where $B \supseteq M$.
  Let $\varphi(\bar x,\bar b)$ be a formula over $B$
  and let $d_p\varphi(\bar y,c)$ be the $\varphi$-definition of $p$,
  $c \in M$.
  Assume that $d_p\varphi(\bar b,c) \neq \varphi(\bar x,\bar b)^q$,
  i.e.,
  $|d_p\varphi(\bar b,c) - \varphi(\bar x,\bar b)|^q > 0$.
  Then there is $\bar b' \in M$ such that
  $|d_p\varphi(\bar b',c) - \varphi(\bar x,\bar b')|^q > 0$,
  a contradiction.

  Conversely, assume $p$ admits a unique heir to every structure.
  let $\cL'$ be $\cL$ along with a new predicate symbol
  $D_\varphi(\bar y)$ for each formula $\varphi(\bar x,\bar y)$
  (here $\bar x$ is fixed, $\bar y$ may vary with $\varphi$).
  We define an expansion $\cM'$ of $\cM$ by interpreting
  $D_\varphi(\bar b) = \varphi(\bar x,\bar b)^p$.
  Assume now that $\cN' \succeq_{\cL'} \cM'$ and let
  \begin{gather*}
    \cN = \cN'\rest_\cL,
    \qquad
    q = \bigl\{ \varphi(\bar x,\bar b) = D_\varphi(\bar b) \bigr\}
    _{\bar b \in N}.
  \end{gather*}
  It is not difficult to see that $q(\bar x) \in \tS_n(N)$ is a
  complete, consistent type, and that it is moreover an heir of $p$.
  Since, by assumption, the heir is unique,
  $\cN'$ is the unique expansion of $\cN$ which is an
  elementary extension of $\cM'$.
  In other words, a model of $\Diag(\cM)$ admits at most one expansion
  to a model of $\Diag(\cM')$.
  By Beth's Theorem (see \cite{BenYaacov:NakanoSpaces})
  for each formula $\varphi(\bar x,\bar y)$ there exists an
  $\cM$-definable predicate $d_p\varphi(\bar y)$ such that
  $\Diag(\cM') \vdash D_\varphi = d_p\varphi$.
  In particular,
  $\varphi(\bar x,\bar m)^p = D_\varphi(\bar m) = d_p\varphi(\bar m)$
  for every $\bar m \in M$, and $p$ is definable.
\end{proof}

Notice that for a pair of models $\cM \subseteq \cN$
we could have defined a notion of a $\varphi$-type over a $\cN$ being
an heir of its restriction to $\cM$,
in which case \fref{lem:UniqueHeir} holds, with the same proof, for
local types.

\begin{thm}
  \label{thm:UniqueHeirCoHeirStable}
  The following are equivalent for a theory $T$:
  \begin{enumerate}
  \item The theory $T$ is stable.
  \item Every type over a model has a unique co-heir to any superset.
  \item Every type over a model has a unique heir to any superset.
  \end{enumerate}
\end{thm}
\begin{proof}
  \begin{cycprf}
  \item[\impnext]
    Assume $T$ is stable, $\cM \subseteq B$,
    and $q \in \tS_n(B)$ is a co-heir of $p = q\rest_M$.
    Let $\cN \supseteq B$ be saturated over $\cM$ and let
    $q' \in \tS_n(N)$ extend $q$, also a co-heir of $p$.
    Then $q'$ is $M$-invariant and therefore the unique non forking
    extension of $p$ to $N$.
    Thus $q$ is the unique non forking extension of $p$ to $B$.
  \item[\impnext]
    Let $\cM$ be a model, $p \in \tS_n(M)$.
    In order to show that $p$ has a unique heir to every
    $B \supseteq M$ it is enough to consider the case
    $B= M\bar b$ where $\bar b$ is a finite tuple.
    So indeed, assume that $\bar a$ realises an heir of $p$ to
    $M\bar b$.
    Then $\tp(\bar b/M\bar a)$ is a co-heir of $\tp(\bar b/M)$ and by
    assumption it is uniquely determined by
    $\tp(\bar b/M)$ and by $\bar a$.
    It follows that $\tp(\bar a/M\bar b)$ is uniquely determined by
    $\bar b$ and $\tp(\bar a/M)$, as desired.
  \item[\impfirst]
    The assumption and \fref{lem:UniqueHeir}
    yield that every type is definable, so $T$ is stable.
  \end{cycprf}
\end{proof}

Using the local version of \fref{lem:UniqueHeir} alluded to above we
can prove a local version of \fref{thm:UniqueHeirCoHeirStable}, namely
that $\varphi(\bar x,\bar y)$ is stable if and only if every
$\varphi$-type over a model admits a unique co-heir to larger sets if
and only if every $\varphi$-type over models admits a unique heir to
larger models.

\section{Invariant types,
  indiscernible sequences and dividing}

\begin{fct}
  \label{fct:NonSplittingSeq}
  Let $\cM$ be a model saturated over $A \subseteq M$,
  and let $p \in \tS_n(M)$ be $A$-invariant.
  Let
  $(\bar a_n)_{n\in\bN} \subseteq M$
  be a sequence constructed inductively,
  choosing each $\bar a_n$ to realise $p\rest_{A\bar a_{<n}}$.

  Then the sequence $(\bar a_n)_{n\in\bN}$
  is $A$-indiscernible, and its
  type over $A$ depends only on $p$.
\end{fct}
\begin{proof}
  Standard.
\end{proof}

The common type over $A$ of such sequences will be denoted by
$p^{(\omega)}\rest_A$.
For every finite or countable $B \subseteq M$ we may construct
$p^{(\omega)}\rest_{A\cup B}$ just as well.
By a gluing argument,
$p^{(\omega)}
= \bigcup \bigl\{
p^{(\omega)}\rest_{A\cup B} \colon B \in [M]^{\aleph_0}
\bigr\}$
is a complete type of an $M$-indiscernible sequence in $p$, and is of
course $A$-invariant.

\begin{lem}
  \label{lem:NonSplitSeqPhiDef}
  Let $A$ be a set, $\varphi(\bar x,\bar y)$ a stable formula,
  $p \in \tS_\varphi(A)$ a stationary $\varphi$-type.
  Let $\cM \supseteq A$ be saturated over $A$,
  and let $p \subseteq q \in \tS_n(M)$, $q$ invariant
  over $A$.
  Let $(\bar c_n)_{n\in\bN} \models q^{(\omega)}\rest_A$
  be an $A$-indiscernible sequence as constructed in
  \fref{fct:NonSplittingSeq}.

  Then the sequence
  $\bigl( \tilde \varphi^n(\bar y,\bar c_{\leq 2n}) \bigr)_{n\in\bN}$
  converges uniformly to the definition
  $d_p\varphi(\bar y)$ at a rate which only
  depends on $\varphi$.
\end{lem}
\begin{proof}
  Since $q\rest_\varphi$ is $A$-invariant, it does not fork over
  $A$, so $d_p\varphi(\bar y) = d_q\varphi(\bar y)$.

  Fix $\varepsilon > 0$.
  By \fref{fct:StabPhiTypeDef} there is $k = k(\varphi,\varepsilon)$
  and a sequence
  $(\bar c'_n)_{n \leq 2k} \subseteq M$ such that
  $|d_p\varphi(\bar y) - \tilde \varphi^k(\bar y,\bar c'_{\leq 2k})|
  \leq \varepsilon$,
  and such that furthermore
  $\bar c'_n \models q\rest_{A,\bar c'_{<n}}$.
  By \fref{fct:NonSplittingSeq} we have
  $\bar c_{\leq 2k} \equiv_A \bar c'_{\leq 2k}$.
  In addition, $d_p\varphi$ is over $A$, so
  $|d_p\varphi(\bar y) - \tilde \varphi^k(\bar y,\bar c_{\leq 2k})|
  \leq \varepsilon$.

  Consider now $n > k$.
  First of all, by exactly the same argument as above,
  for every $w \in [2n+1]^{2k+1}$ we have
  $|d_p\varphi(\bar y) - \tilde \varphi^k(\bar y,\bar c_{\in w})|
  \leq \varepsilon$.
  In addition, for any $\bar b$ there exists a subset
  $w \in [2n+1]^{2k+1}$ such that 
  $\tilde \varphi^n(\bar b,\bar c_{\leq 2n}) =
  \tilde \varphi^k(\bar b,\bar c_{\in w})$
  (from any set of $2n+1$ reals one can choose a subset of size
  $2k+1$ with the same median value).
  Thus
  $|d_p\varphi(\bar y) - \tilde \varphi^n(\bar y,\bar c_{\leq 2n})|
  \leq \varepsilon$
  for all
  $n \geq k$, where $k$ depends only on $\varepsilon$ and $\varphi$,
  as desired.
\end{proof}

\begin{prp}
  \label{prp:NonDividing}
  Let $\varphi(\bar x,\bar b)$ be an instance of a stable formula,
  $A$ a set of parameters.
  Then the following are equivalent:
  \begin{enumerate}
  \item The condition $\varphi(\bar x,\bar b) \leq r$ does not fork over $A$.
  \item
    If $(\bar b_n)_{n\in\bN}$ is an $A$-indiscernible sequence,
    $\bar b_0 = \bar b$, then the set of conditions
    $\{\varphi(\bar x,\bar b_n) \leq r\}_{n\in\bN}$ is consistent
    (i.e., the condition
    $\varphi(\bar x,\bar b) \leq r$ does not \emph{divide}
    over $A$).
  \end{enumerate}
\end{prp}
\begin{proof}
  \begin{cycprf}
  \item[\impnext]
    If $(\bar b_n)_{n\in\bN}$ is an $A$-indiscernible sequence and
    $\bar b_0 = \bar b$ then each $\bar b_n$ is an
    $\acl^{eq}(A)$-conjugate of $\bar b$.
  \item[\impfirst]
    Fix models $\cN \succeq \cM \supseteq A$ where $\cN$ is saturated
    over $M$.
    Let $q_0 = \tp(\bar b/\acl^{eq}(A))$.
    By \fref{lem:NFExtExist} there exists
    $q \in \tS_m(M)$ extending $q_0$ such that
    $q\rest_{\tilde \varphi}$ does not fork over $A$,
    i.e., such that
    $d_q \tilde \varphi = d_{q_0} \tilde \varphi$.
    Let $q_1 \in \tS_m(N)$ be an $M$-invariant extension of $q$.
    Finally, let
    $(\bar b_n)_{n\in\bN} \models q_1^{(\omega)}\rest_M$.
    Then $(\bar b_n)_{n\in\bN}$ is an $\cM$-indiscernible sequence,
    and \emph{a fortiori} $A$-indiscernible, in $\tp(\bar b/A)$.
    Thus by assumption there exists $\bar a$ such that
    $\varphi(\bar a,\bar b_n) \leq r$ for all $n$.
    In addition, by \fref{lem:NonSplitSeqPhiDef} we have
    \begin{gather*}
      d_q\tilde \varphi(\bar a)
      =
      \lim_n \med_n \bigl( \varphi(\bar a,\bar b_i) \bigr)_{i\leq 2n}
      \leq r.
    \end{gather*}
    Let $p \in \tS_\varphi(M)$ be a non forking extension
    of $\tp_\varphi(\bar a/\acl^{eq}(A))$.
    Then
    $\varphi(\bar x,\bar b)^p = d_q \tilde \varphi(\bar x)^p \leq r$,
    witnessing that
    $\varphi(\bar x,\bar b) \leq r$ does not fork over $A$,
    as desired.
  \end{cycprf}
\end{proof}

\section{Canonical bases}
\label{sec:CanonicalBases}

Recall that the \emph{canonical base} of a stationary
type $p \in \tS_n(A)$ in a stable theory is 
$\Cb(p)
= \{\Cb(p\rest_\varphi)\colon \varphi(\bar x,\ldots) \in \cL\}$,
namely the set of all
canonical parameters of $\varphi$-definitions of $p$.

\begin{prp}
  \label{prp:CanonicalBase}
  Assume $T$ is stable, and let $p(\bar x) \in \tS_n(A)$ be
  stationary.
  Then:
  \begin{enumerate}
  \item $\Cb(p) \subseteq \dcl^{eq}(A)$.
  \item $p$ does not fork over $\Cb(p)$.
  \item $p\rest_{\Cb(p)}$ is stationary.
  \item $\Cb(p)$ is minimal for the three previous properties,
    meaning that if $B \subseteq \dcl^{eq}(A)$ and $p\rest_B$ is a stationary non
    forking restriction then $\Cb(p) \subseteq \dcl^{eq}(B)$.
  \end{enumerate}
\end{prp}
\begin{proof}
  The first two items are immediate, while the third is by
  \fref{cor:DefinableStationary}.
  Under the assumptions of the fourth we have
  $\Cb(p) = \Cb(p\rest_B) \subseteq \dcl^{eq}(B)$.
\end{proof}

The four properties listed in \fref{prp:CanonicalBase} determine the
canonical base up to inter-definability.
Indeed, if $B$ has all four then $\Cb(p) \subseteq \dcl^{eq}(B)$ but
also $B \subseteq \dcl^{eq}(\Cb(p))$,
whereby $\dcl^{eq}(B)= \dcl^{eq}(\Cb(p))$.
In this case we say that $B$ is \emph{a} canonical base for $p$.

\begin{prp}
  Assume $T$ is stable, and let $p(\bar x) \in \tS_n(A)$ be
  stationary.
  Let $q \in \tS_n(M)$ be the unique non forking extension of $p$,
  where $\cM$ is saturated over $A$.
  Then a (small) set $B \subseteq M$ is a canonical base for $p$ if and
  only if, for every $f \in \Aut(\cM)$:
  $f\rest_B = \id_B \Longleftrightarrow f(q) = q$.
\end{prp}
\begin{proof}
  Let $C = \Cb(p) = \Cb(q)$.
  It follows directly from the definitions that an automorphism of
  $\cM$ fixes $q$ if and only if it fixes $q\rest_\varphi$ for every
  formula $\varphi(\bar x,\ldots)$, if and only if it fixes every
  member of $C$.
  A small set $B$ is another canonical base for $p$ if and
  only if $\dcl^{eq}(B) = \dcl^{eq}(C)$ which is further equivalent to
  $B$ and $C$ being fixed by the same automorphisms.
\end{proof}

We propose an alternative characterisation of canonical bases using
Morley sequences.
In the case of classical first order logic it is more or less
folklore.
Recall that a \emph{Morley sequence} in a (stationary) type
$p(\bar x) \in \tS_m(A)$
is a sequence $I = (\bar a_n)_{n\in\bN}$ of realisations of $p$
which is independent over $A$, i.e.,
such that $\bar a_n \ind_A \bar a_{<n}$ for all $n \in \bN$.
It follows by standard independence calculus that
$\bar a_{\in s} \ind_A \bar a_{\in t}$
for every two disjoint index sets $s,t \subseteq \bN$.
From stationarity of $p$ it follows that the sequence
$I$ is indiscernible over $A$, and its type over
$A$, which we may denote by $p^{(\omega)}$,
is uniquely determined by $p$.

It is not difficult to check that if $p$ satisfies the assumptions of
\fref{fct:NonSplittingSeq} then the definition of $p^{(\omega)}$ which
appears thereafter agrees with the one given here.
In the general case, let $\cM$ be saturated over $A$ and let
$q \in \tS_m(M)$ be the non forking extension of $p$.
Then by construction, $p^{(\omega)} = q^{(\omega)}_A$,
where the first is the type of a Morley sequence as defined here,
and the second the type defined after \fref{fct:NonSplittingSeq}.

\begin{dfn}
  Let $I = (\bar a_n)_{n\in\bN}$ be a sequence of tuples,
  or, for that matter, even of sets.
  Let $I^{\geq k}$ denote the tail $(\bar a_n)_{n \geq k}$.
  We define the \emph{tail definable closure} of $I$ as
  \begin{gather*}
    \tdcl^{eq}(I) =
    \bigcap_{k \in \bN} \dcl^{eq}(I^{\geq k}).
  \end{gather*}
\end{dfn}

It is not difficult to see that for an indiscernible sequence $I$,
$\tdcl^{eq}(I)$ consists precisely of all $c \in \dcl^{eq}(I)$ over
which $I$ is indiscernible.

\begin{lem}
  Let $I = (\bar a_n)_{n\in\bN}$ and $J = (\bar b_n)_{n\in\bN}$
  be indiscernible sequences such that the concatenation
  $I \concat J$ is indiscernible as well.
  Then $\tdcl^{eq}(I) = \tdcl^{eq}(J)$.
  Moreover, every automorphism which sends $I$ to $J$
  necessarily fixes $\tdcl^{eq}(I)$.
\end{lem}
\begin{proof}
  For $k \in \bN$ let
  $J_k$ be the sequence
  $\bar a_0,\ldots,\bar a_{k-1},\bar b_k,\bar b_{k+1},\ldots$,
  namely the sequence obtained by replacing the first $k$ elements of
  $J$ with the corresponding elements from $I$.
  Since $I \concat J$ is indiscernible so is $J_k$ for each $k$,
  and there exists an automorphism $f_k$ sending $J \mapsto J_k$.
  Now let $c \in \tdcl^{eq}(J)$.
  Since $c$ is definable over $J^{\geq k}$
  it is fixed by $f_k$, so $cJ \equiv  cJ_k$.
  This holds for all $k$, whence $cI \equiv cJ$.

  Fix an automorphism $f$ which sends $I$ to $J$
  (which must necessarily exist).
  Then $f(c)J \equiv cI \equiv cJ$, so $f(c) = c$.
  Thus $f$ fixes $\tdcl^{eq}(J)$.
  Applying $f^{-1}$ we obtain that
  $\tdcl^{eq}(I) = \tdcl^{eq}(J)$, as desired.
\end{proof}

\begin{thm}
  \label{thm:TailCB}
  Let $p \in \tS_m(A)$ be a stationary type and let
  $I = (\bar a_n)_{n\in\bN}$ be a Morley sequence in $p$.
  Then $\tdcl^{eq}(I)$ is a canonical base of $p$.
\end{thm}
\begin{proof}
  First of all, we have seen that $p\rest_{\Cb(p)}$ is
  stationary, with the same canonical base as $p$.
  It is also not difficult to check that a Morley sequence in $p$ is
  also a Morley sequence in $p\rest_{\Cb(p)}$.
  It is therefore enough to prove for $p\rest_{\Cb(p)}$, i.e.,
  we may assume that $A = \Cb(p)$.

  So let $\cM$ be saturated over $A = \Cb(p)$
  and let $q \in \tS_m(M)$ be the non
  forking extension of $p$.
  As pointed above,
  $I \models p^{(\omega)} = q^{(\omega)}_A$.
  By \fref{lem:NonSplitSeqPhiDef} $p$ is definable over $I$, so
  $\Cb(p) \subseteq \dcl^{eq}(I)$.
  Also, every tail of a Morley sequence is a Morley sequence,
  whence $\Cb(p) \subseteq \tdcl^{eq}(I)$.

  Conversely, let $f$ be an automorphism fixing $A = \Cb(p)$.
  Then $f$ fixes $p$ and therefore sends $I$ to another Morley
  sequence in $p$, say $J$.
  Let $K$ be a third Morley sequence in $p$, $K \ind_A I,J$.
  Then both $I \concat K$ and $J \concat K$ can be verifies to be
  Morley sequences in $p$ (of length $\omega + \omega$), and in
  particular indiscernible.
  We can decompose $f = h \circ g$
  where $g(I) = K$ and $h(K) = J$.
  By the Lemma
  $\tdcl^{eq}(I) = \tdcl^{eq}(K) = \tdcl^{eq}(J)$
  and this set is fixed by $g$, $h$ and therefore by $f$.
  Thus $\tdcl^{eq}(I) \subseteq \dcl\bigl(\Cb(p)\bigr)$, and the proof
  is complete.
\end{proof}

It is also a fact, which we shall not prove here (but is proved as in
classical logic), that in a stable theory every indiscernible sequence
$I = (\bar a_n)_{n\in\bN}$ is a Morley sequence in some type, say $q$.
Let $A = \tdcl^{eq}(I)$ and
$p = \tp(\bar a_n/A)$, which does not depend on $n$.
By the Theorem, $A = \Cb(q) = \Cb(p)$ and $I$ is a Morley sequence in
$p$.

In the case of probability theory this is a well known fact.
Indeed, in probability algebras or in spaces of random variables
(say $[0,1]$-valued, see \cite{BenYaacov:RandomVariables}),
the canonical base of a type
(in the real sort) can be represented by a set of real elements, so
there is no need to consider imaginaries.
Then \fref{thm:TailCB} tells us
that if $(X_n)_{n\in\bN}$ is sequence of random variables
which is indiscernible (i.e., exchangeable)
and $\sA$ is its tail algebra then
the sequence $(X_n)_{n\in\bN}$ is i.i.d.\ over $\sA$,
meaning that the random variables $X_n$ are independent over $\sA$
and have the same conditional distribution over $\sA$.

\begin{cor}
  Assume $T$ is stable, and let $p(\bar x) \in \tS_m(A)$ be stationary.
  Let $I = (\bar a_n)_{n\in\bN}$ be a Morley sequence in $p$,
  $J = I \setminus \bar a_0$.
  Then $\bar a_0 \ind_A J$ and $\bar a_0 \ind_J A$.
\end{cor}
\begin{proof}
  The first independence is immediate and implies
  $\bar a_0 \ind_{\Cb(p)} AJ$.
  By \fref{thm:TailCB}  we have
  $\Cb(p) \subseteq \dcl^{eq}(J)$ and
  the second independence follows.
\end{proof}

\section{Stable type-definable groups and their actions}

We turn to consider groups, and more generally, homogeneous spaces,
which are definable or type-definable in a stable theory.

\subsection{Generic elements and types in stable group actions}

Let $\langle G,S\rangle$ be a homogeneous space, type-definable in models of a
stable theory $T$.
This is to say that $G$ is a
type-definable group and $S$ a type-definable set, equipped with a
type-definable (and therefore definable) transitive
group action $G \times S \to S$.
For convenience let us assume that both are defined without
parameters.
We shall identify $G$ and $S$ with their sets of realisations in a
monster model $\fM$.
We are particularly interested in the case where $S = G$ where $G$
acts on itself either on the left $(g,h) \mapsto gh$ or on the right
$(g,h) \mapsto hg^{-1}$.

Given a partial type $\pi(x)$ in the sort of $S$ we
let $\pi(S)$ denote the subset of $S$ defined by $\pi$.

\begin{dfn}
  \begin{enumerate}
  \item 
    A \emph{generic set} in $S$ is a subset $X \subseteq S$
    finitely many $G$-translates of which cover $S$.
  \item
    A \emph{generic partial type} in $S$ is a partial type $\pi(x)$
    such that every logical neighbourhood of $\pi$
    (as per \fref{dfn:LogNeighb}) defines in $S$ a generic set.
    Single conditions as well as complete types are generic if they
    are generic as partial types.
  \item We say that an element $s \in S$ is \emph{generic} over a set
    $A$ if $\tp(s/A)$ is generic.
  \item
    A \emph{left-generic set} in $G$ is a subset $X \subseteq G$
    which is generic under the
    action of $G$ on itself on the left.
    We define partial types
    in the sort of $G$ to be
    left-generic accordingly.
    Similarly for right-generic.
  \end{enumerate}
\end{dfn}

Let $\pi(x)$ be a partial type.
Clearly, if $\pi(S)$ is a generic set then $\pi$ is a generic partial
type, but the converse is not always true.
In classical logic, if $\pi$ consists of a single formula
(i.e., if $\pi(S)$ is a relatively definable subset of $S$, and so is
its complement),
then $\pi$ is its own logical neighbourhood and the two notions
coincide.
Unfortunately, this will generally never happen in continuous logic
(except for $\pi(S) = S$ or $\pi(S) = \emptyset$).

\begin{lem}
  \label{lem:GenPartialType}
  The following are equivalent for a partial type $\pi(x)$
  in the sort of $S$, with parameters in a set $A$:
  \begin{enumerate}
  \item The partial type $\pi$ is generic in $S$.
  \item For every formula $\varphi(x,\bar a)$ over $A$,
    if the condition $\varphi(x,\bar a) = 0$ is a logical
    neighbourhood of $\pi$ then it is a generic condition.
  \end{enumerate}
\end{lem}
\begin{proof}
  One direction is immediate, the other follow from
  \fref{lem:LogNeighb}.
\end{proof}

Let $\tS_S(A)$ denote the set of all complete types over
$A$ implying $x \in S$.
Equipped with the induced topology from $\tS_x(A)$,
it is a compact space, and the set of all generic complete
types over $A$ is closed.
Closed subsets of $\tS_S(A)$ are in bijection with partial types over
$A$ implying $x \in S$, i.e., with type-definable subsets of $S$ using
parameters in $A$.
If $X,Y \subseteq S$ are two such sets, say that $Y$ is a logical
neighbourhood of $X$ \emph{relative to $S$}, in symbols $Y >^S X$, if
$[X] \subseteq [Y]^\circ$ where the interior is calculated in $\tS_S(A)$.
This is equivalent to saying that there exists a true logical
neighbourhood $Y' > X$ such that $Y = Y' \cap S$.
Thus a type-definable set $X \subseteq S$ is defined by a generic
partial type in $S$ if and only if
every relative logical neighbourhood of $X$ in $S$ defines a generic
set.

For $g \in G$ and $X \subseteq S$,
let $L_g[X] = gX = \{gs\}_{s \in X}$.
Somewhat superfluously, we may also define
$L_g^{-1}[X] = \{s \in S\colon gs \in X\} = L_{g^{-1}}[X]$.

\begin{lem}
  \label{lem:GenericTypeLeftAction}
  Let $A$ be a set of parameters,
  $g \in G(A) = G \cap \dcl(A)$.
  \begin{enumerate}
  \item
    If $X \subseteq S$ is type-definable over $A$, say by a partial
    type $\pi$, then $L_g[X]$ is also type-definable over $A$ by a
    partial type which will be denoted $L_g\pi$ (or $g\pi$).
    Moreover, $\pi$ is generic if and only if $g\pi$ is.
  \item
    If $p = \tp(s/A) \in \tS_S(A)$ is a complete type
    then $L_gp = gp = \tp(gp/A)$,
    and $L_g\colon \tS_S(A) \to \tS_S(A)$
    is a homeomorphism, and restricts to a homeomorphism of the set of
    generic types with itself.
  \end{enumerate}
\end{lem}
\begin{proof}
  We only prove the parts regarding genericity.
  Indeed, assume that $\pi$ is generic, and let
  $gX <^S Y$.
  Then $X <^S L_g^{-1}[Y]$, so $L_g^{-1}[Y]$ is a generic subset of
  $S$.
  It follows immediately that so is $Y$.
  Thus $g\pi$ is a generic partial type.
  For the converse replace $g$ with $g^{-1}$.
\end{proof}

Similarly, for $s \in S$ and $X \subseteq G$ we define
$R_s[X] = Xs = \{gs\}_{g \in X}$.
For $X \subseteq S$ we define
$R_s^{-1}[X] = \{g \in G\colon gs \in X\}$.

\begin{lem}
  \label{lem:GenericTypeRightAction}
  Let $A$ be a set of parameters,
  $s \in S(A) = S \cap \dcl(A)$.
  \begin{enumerate}
  \item
    If $X \subseteq G$ is type-definable over $A$, say by a partial
    type $\pi$, then $R_s[X]$ is also type-definable over $A$ by a
    partial type which will be denoted $R_s\pi$ (or $\pi s$).
    Moreover, if $\pi$ is left-generic then $R_s\pi$ is generic.
  \item
    If $p = \tp(g/A) \in \tS_G(A)$ is a complete type
    then $R_sp = ps = \tp(gs/A)$,
    and $R_s\colon \tS_G(A) \to \tS_S(A)$
    is a continuous surjection, sending left-generic types to generic
    types.
  \end{enumerate}
  Notice that we do not claim that every generic type in $\tS_S(A)$ is
  the image under $R_s$ of a left-generic type in $\tS_G(A)$ (this is
  true if $T$ is stable).
\end{lem}
\begin{proof}
  Essentially identical to that of \fref{lem:GenericTypeLeftAction}.
\end{proof}

Under the assumption that the theory $T = \Th(\fM)$ is stable we shall
show that generic types exist and study some of their properties.
We follow a path similar to that followed in
\cite{Pillay:GeometricStability}.
Toward this end we construct
an auxiliary  multi-sorted structure
$\hat \fM = \langle G,S,\ldots\rangle$
in a language $\hat \cL$
(in addition to sorts $G$ and $S$, $\hat \cL$ consists of additional
sorts which we shall described later).
We define the distance on the first two sorts by
\begin{gather*}
  d_G^{\hat \fM}(g,g') = \sup_{h \in G} \, d^\fM(hg,hg'),
  \qquad
  d_S^{\hat \fM}(s,s') = \sup_{h \in G} \, d^\fM(hs,hs').
\end{gather*}
This coincides with the original distance in $\fM$ if the latter is
invariant under the action of $G$ (on the left).
In any case, $d_G^{\hat \fM}$ is a distance function, invariant under
the action of $G$, and satisfies $d_G^{\hat \fM} \geq d^\fM$.
On the other hand, if $g_n \to g$ in $d^\fM$
then $g_n \to g$ in $d_G^{\hat \fM}$ as well
(if not, then by a compactness argument,
for some $\varepsilon > 0$ there would exist $h \in G$
such that $d^\fM(hg,hg) \geq \varepsilon$, an absurd).
It follows that $(G,d_G^{\hat \fM})$ is a complete metric space.
The same observations hold for $(S,d_S^{\hat \fM})$.

Let now $\Phi_S$ consist of all $\cL$-formulae of the form
$\varphi(x,\bar y)$ where $x$ is in the sort of $S$.
For each $\varphi \in \Phi_S$, there will be a sort
$C_\varphi$, consisting of all canonical parameters of instances of
$\varphi$ in $\fM$.
The canonical parameter of $\varphi(x,\bar b)$ will be denote $[\bar b]_\varphi$,
or $[\bar b]$ if there is no ambiguity.
We put on it the standard metric, namely
\begin{gather*}
  d_\varphi([\bar b]_\varphi,[\bar b']_\varphi)
  =
  \sup_{a \in \fM}\,
  |\varphi(a,\bar b)-\varphi(a,\bar b')|.
\end{gather*}
The only symbols in the language $\hat \cL$, in addition to the
distance symbols of the various sorts, are a predicate symbol
$\hat \varphi(x_S,y_G,z_\varphi)$ for each formula $\varphi \in \Phi_S$,
interpreted by
\begin{gather*}
  \hat \varphi(s,g,[\bar b])^{\hat \fM}
  =
  \varphi(g^{-1}s,\bar b)^\fM.
\end{gather*}
Since $\varphi$ is uniformly continuous in all its variables,
so is $\hat \varphi$.
These definitions make $\hat \fM$ a continuous
$\hat \cL$-structure.

If $\langle G,S\rangle$ is definable then $\hat \fM$
is interpretable in $\fM$ and
$\hat T = \Th_{\hat \cL}(\hat \fM)$ is stable (assuming $T$ is).
In the general case, all we know is that $\hat \fM$ is saturated
for quantifier-free types in which only $\hat \varphi$ appear.
It follows from stability in $T$ that each formula $\hat \varphi(x,y,z)$,
with any partition of the variables, is stable.

For $h \in G$ define a mapping $\theta_h\colon \hat \fM \to \hat \fM$
by sending $g \in G$ to $hg$, $s \in S$ to $hs$, and fixing all
the auxiliary sorts.
This is easily verified to be an automorphism of $\hat \fM$.
Since the action of $G$ on $S$ is assumed to be transitive,
if $A \subseteq \bigcup_\varphi C_\varphi$ then all elements of $S$
have the same type over $A$
in $\hat \fM$, and similarly all elements of $G$.

\begin{lem}
  \label{lem:GenSetNF}
  Assume that $\varphi(x,\bar y) \in \Phi_S$ is stable.
  Then the following are equivalent for an instance
  $\varphi(x,\bar b)$:
  \begin{enumerate}
  \item The condition $\varphi(x,\bar b) = 0$ is generic in $S$.
  \item The condition $\hat \varphi(x,e,[\bar b]) = 0$ does not fork in
    $\hat \fM$ over $\emptyset$.
  \item The condition $\hat \varphi(x,e,[\bar b]) = 0$ does not fork in
    $\hat \fM$ over $[\bar b]$.
  \end{enumerate}
\end{lem}
\begin{proof}
  Recall that the $\hat \cL$-formula $\hat \varphi(x_S,y_Gz_\varphi)$
  with this
  (or any other) partition of the variables is stable in $\hat \fM$.
  For $\varepsilon > 0$ let
  $X_\varepsilon
  = \{s \in S\colon \varphi(s,\bar b) \leq \varepsilon\}$.
  By \fref{lem:GenPartialType},
  the condition $\varphi(x,\bar b) = 0$ is generic if and only if
  $X_\varepsilon$ is a generic set for all $\varepsilon > 0$.
  \begin{cycprf}
  \item[\impnext]
    Assume first that $\varphi(x,\bar b) = 0$ is generic in $S$, i.e.,
    that the set $X_\varepsilon$ is generic for every $\varepsilon > 0$.
    Find $g_i \in G$ such that $S = \bigcup_{i<n} g_iX_\varepsilon$, and
    find $s \in S$ such that
    $\tp_{\hat \varphi}(s/[\bar b]g_{<n})$ does not fork over $\emptyset$
    (in symbols $s \ind[\hat \varphi] \,[\bar b]g_{<n}$).
    Since $s \in \bigcup_{i<n} g_iX_\varepsilon$ we may assume that
    $s \in g_0X_\varepsilon$,
    so
    $\hat \varphi(s,g_0,[\bar b])
    = \varphi(g_0^{-1}s,\bar b) \leq \varepsilon$.
    Thus $\hat \varphi(x,g_0,[\bar b]) \leq \varepsilon$
    does not fork over $\emptyset$.
    Applying $\theta_{g_0^{-1}}$ we see that
    $\hat \varphi(x,e,[\bar b]) \leq \varepsilon$ does
    not fork over $\emptyset$ either.
    It follows that
    $\hat \varphi(x,e,[\bar b]) = 0$
    does not fork over $\emptyset$.
  \item[\impnext]
    Immediate.
  \item[\impfirst]
    Assume now that
    $\hat \varphi(x,e,[\bar b]) = 0$ does not fork over $[\bar b]$.
    By \fref{prp:FaithfulNF} there are $g_n \in G$
    for $n \in \bN$
    and a faithful combination
    $\psi(x,[\bar b])
    = \theta\bigl( \hat \varphi(x,g_n,[\bar b]) \bigr)_{n\in\bN}$
    which is definable over $[\bar b]$ and such that
    $\psi(x,[\bar b]) = 0$ is consistent.
    Since $\hat \fM$ is saturated for quantifier-free types involving only
    $\hat \varphi$, there is $s \in S$ such that
    $\psi(s,[\bar b]) = 0$.
    Since all elements of $S$ have the
    same type over $[\bar b]$ in $\hat \fM$, we see that
    $\psi(s,[\bar b]) = 0$ for all $s \in S$.
    Assume (toward a contradiction) that there exists $\varepsilon > 0$
    such that $\varphi(x,\bar b) \leq \varepsilon$ is not generic.
    By compactness we can find $s \in S$ such that
    $\varphi(g_n^{-1}s,\bar b) \geq \varepsilon$
    for all $n$, i.e.,
    $\hat \varphi(s,g_n,[\bar b]) \geq \varepsilon$.
    Since the combination above was faithful we get
    $\psi(s,[\bar b]) \geq \varepsilon > 0$, a contradiction.
  \end{cycprf}
\end{proof}

\begin{lem}
  \label{lem:GenCondNF}
  Assume that $\varphi(x,\bar y) \in \Phi_S$ is stable
  and that $\varphi(x,\bar b) = 0$ is a generic condition in $S$.
  Then it does not fork over $\emptyset$.
\end{lem}
\begin{proof}
  By \fref{prp:NonDividing} it will be enough to show that
  $\varphi(x,\bar b) = 0$ does not divide over $\emptyset$.
  For this purpose let $(\bar b_n)_{n\in\bN}$ be any indiscernible
  sequence with $\bar b_0 = \bar b$.
  Since $e \in \dcl(\emptyset)$, the sequence
  $(e,\bar b_n)_{n\in\bN}$ is
  indiscernible as well, and thus
  the sequence $(e,[\bar b_n])_{n\in\bN}$ is indiscernible in
  $\hat \fM$.
  On the other hand, since the condition
  $\varphi(x,\bar b) = 0$ is generic,
  by \fref{lem:GenSetNF} the condition $\hat \varphi(x,e,[\bar b]) = 0$
  does not fork over $\emptyset$, so
  $\{\hat \varphi(x,e,[\bar b_n])\}_{n\in\bN}$ is consistent.
  Since $\hat \fM$ is saturated for such formulae, there is $s \in S$
  such that $\hat \varphi(s,e,[\bar b_n]) = 0$,
  i.e., $\varphi(s,\bar b_i) = 0$, for all $n$,
  as desired.
\end{proof}

From now on we assume that $T$ is stable.

\begin{prp}
  \label{prp:GenTypeExt}
  Let $\pi(x)$ be a partial type over $A$.
  Then $\pi$ is generic if and only if it extends to a complete
  generic type over $A$, i.e., if and only if
  $[\pi] \subseteq \tS_S(A)$
  contains a generic type.
  In particular, generic types exist over every set.
\end{prp}
\begin{proof}
  Right to left is clear, so let us prove left to right.
  Assume therefore that $\pi$ is a generic partial type.
  Since the set of complete generic types is closed it will be enough
  to show that every logical neighbourhood of $\pi$ contains a generic
  type, and we may further restrict our attention to logical
  neighbourhoods defined by a single condition
  $\varphi(x,\bar b) = 0$.
  Since $\pi$ is generic in $S$ so is $\varphi(x,\bar b) = 0$.
  By \fref{lem:GenSetNF} $\hat \varphi(x,e,[\bar b]) = 0$ does not
  fork over $\emptyset$ in $\hat \fM$.
  By \fref{cor:CompleteNFExt} there exists a type
  $\hat p \in \tS_x(\hat \fM)$
  such that
  $\hat \varphi(x,e,[\bar b])^{\hat p} = 0$
  and in addition
  $p\rest_{\hat \psi}$ does not fork over $\emptyset$
  for every formula $\psi \in \Phi_S$.
  Let
  \begin{gather*}
    p(x)
    = \bigl\{
    \psi(x,\bar c) = \hat \psi(x,e,[\bar c])^{\hat p}
    \bigr\}_{\psi \in \Phi_S, \bar c \in \fM}.
  \end{gather*}
  This type is approximately finitely realised in $\fM$
  (since $\hat p$ is in $\hat \fM$) and therefore consistent.
  By \fref{lem:GenSetNF} every condition in $p$ is generic
  (since $\hat p$ does not fork over $\emptyset$),
  and by \fref{lem:GenPartialType}, $p$ is generic,
  and so is $p\rest_A$.
  We have thus found a generic type
  $p\rest_A \in [\varphi(x,\bar b) = 0]$ and the proof is complete.
\end{proof}

\begin{prp}
  \label{prp:GenTypeNFExt}
  Assume $A \subseteq B$.
  Then a type $p \in \tS_S(B)$ is generic if and only if it does not
  fork over $A$ and $p\rest_A$ is generic.
  In particular, a generic type does not fork over $\emptyset$.  
\end{prp}
\begin{proof}
  First of all, the last assertion follows from \fref{lem:GenCondNF}
  and the fact that the set of non forking types is closed.

  We now prove the main assertion.
  For left to right, if $p \in \tS_S(B)$ is generic
  then clearly so is $p\rest_A$, and by the previous paragraph $p$
  does not fork over $A$.
  For the converse, assume that $p \in \tS_S(B)$ does not fork over $A$
  and $p_0 = p\rest_A$ is generic.
  Replacing $p$ with a non forking extension we may assume that
  $B = \fM$.
  By \fref{prp:GenTypeExt} there is $p_1 \in \tS_S(\fM)$ extending
  $p_0$ which is generic, and by what we have just shown it is
  also non forking over $A$.
  Since $p\rest_A = p_0 = p_1\rest_A$ there is $f \in \Aut(\fM/A)$
  sending $p_1\rest_{\acl^{eq}(A)}$ to $p\rest_{\acl^{eq}(A)}$,
  and therefore $p_1$ to $p$.
  Thus $p$ is generic as well.
\end{proof}

We can also complement \fref{lem:GenericTypeLeftAction}:
\begin{prp}
  The action of $G$ on the set of generic types
  in $\tS_S(\fM)$ is transitive.
\end{prp}
\begin{proof}
  Let $p,q \in \tS_S(\fM)$ be two generic types.
  Define
  \begin{gather*}
    \hat p =
    \{
    \hat\varphi(x,g,[\bar b]) = \varphi(g^{-1}x,[\bar b])^p
    \}_{\varphi \in \Phi_S, \bar b \in \fM, g \in G},
  \end{gather*}
  and define $\hat q$ similarly.
  Let $\hat C = \bigl( \acl^{eq}(\emptyset) \bigr)^{\hat \fM}$ and
  let
  $\hat p_0  = \hat p\rest_C$,
  $\hat q_0  = \hat q\rest_C$.
  Since $\hat \fM$ is saturated for formulae of this form we may
  realise $\hat p_0$ and $\hat q_0$ in $\hat \fM$,
  and by transitivity there exists $h \in G$
  such that
  $\theta_h\hat p_0 \cup \hat q_0$ is realised.
  Since $\theta_h$ is an automorphism of $\hat \fM$
  we must have
  $\hat q_0 = \theta_h\hat p_0 = (\theta_h\hat p)\rest_C$.
  In addition, neither of
  $\hat p$, $\hat q$ or $\theta_h \hat p$ forks over
  $\emptyset$, whereby $\theta_h\hat p = \hat q$,
  i.e., $hp = q$.
\end{proof}

\begin{thm}
  \label{thm:GenTypeProps}
  Let $G$ be a type-definable group in a stable theory
  $T$, acting type-definably and transitively on a
  type-definable set $S$.
  \begin{enumerate}
  \item \label{item:GenTypeProps1}
    If $g \ind_A s$ (where $g \in G$, $s \in S$)
    and $g$ is left-generic over $A$ then $gs$ is
    generic over $A$ and $gs \ind_A s$.
  \item An element $s \in S$ is generic if and only if
    $g \ind_A s$ implies $gs \ind_A g$
    for every $g \in G$.
    Moreover, in this case $gs$ is generic over $A$ as well.
  \item An element $g \in G$ is left-generic over $A$ if and only if
    $g^{-1}$ is.
  \item An element $g \in G$ is left-generic if and only if it is
    right-generic (over $A$).
    From now on we shall only speak of \emph{generic} elements and
    types in $G$.
  \item An element $g \in G$ is generic over $A$ if and only if it is
    generic over $\emptyset$ and $g \ind A$.
  \end{enumerate}
\end{thm}
\begin{proof}
  We use \fref{prp:GenTypeNFExt} repeatedly.

  For the first item, let $s \in S$, $g \in G$,
  and assume that $g \ind_A s$.
  If $g$ is left-generic over $A$ then it is left-generic over $A,s$.
  By \fref{lem:GenericTypeRightAction} $gs$ is generic over $A,s$.
  It follows that $gs$ is generic over $A$ and that $gs \ind A,s$, as
  desired.

  For the second item, left to right, as well as the moreover part,
  are proved as in the previous argument, using
  \fref{lem:GenericTypeLeftAction}.
  For right to left, assume that
  $s \ind_A g$ implies $gs \ind A,g$ for all $g$.
  We may choose $g$ which is left-generic over $A$ such that
  $g \ind_A s$.
  Then $g^{-1} \ind_A gs$ by assumption, $gs$ is generic over $A$ by
  the first item, and $s = g^{-1}gs$ is generic
  over $A$ by the moreover part.

  For the third item, let $g \in G$ be left-generic over $A$.
  Choose $h \in G$ left-generic over $A$ such that $g \ind_A h$.
  By the first item $gh$ is generic over $A$ and
  $gh \ind_A h$.
  This can be re-written as $h \ind_A h^{-1}g^{-1}$.
  By the first item again, $g^{-1} = hh^{-1}g^{-1}$ is left-generic
  over $A$.
  Notice that $g^{-1}$ is left-generic if and only if $g$ is
  right-generic, yielding the fourth item as well.

  The last item is just \fref{prp:GenTypeNFExt}.
\end{proof}

\subsection{Stabilisers}

We have already observed in \fref{lem:GenericTypeLeftAction} that
for any set of parameters $A$, a group element $g \in G(A)$
induces a homeomorphism $L_g\colon p \mapsto gp$ on
$\tS_S(A)$.
It is also not difficult to check that
$L_g \circ L_h = L_{gh}$, whence a group action of $G(A)$ on
$\tS_S(A)$.
In addition, we have seen that it restricts to an action
by homeomorphism of $G(A)$ on the set of generic types in $\tS_S(A)$.

Specifically, we obtain an action of $G = G(\fM)$ on
$\tS_S(\fM)$.
The \emph{stabiliser} of a type $p \in \tS_S(\fM)$
under this action is
$\Stab(p) = \{g\in G\colon gp = p\} \leq G$.
For a stationary type $p \in \tS_S(A)$ we define
$\Stab(p) = \Stab(p\rest^\fM)$.

\begin{prp}
  Let $p \in \tS_S(A)$ be stationary.
  Then stabiliser $\Stab(p)$ is a sub-group of $G$
  type-definable over $\Cb(p)$.

  Moreover, assume that $s \models p$,
  $g \in G$ and $g \ind_A s$.
  Then $g \in \Stab(p)$ if and only if $gs \models p$.
\end{prp}
\begin{proof}
  We may assume that $p \in \tS_S(\fM)$.

  Let $\varphi(x,\bar z)$ be a formula, $x$ in the sort of $S$.
  Let $y$ be a variable in the sort of $G$.
  Then $\varphi(yx,\bar z)$ is a definable predicate on
  $G \times S \times \langle \text{sort of }\bar z\rangle$,
  i.e., a continuous function
  $\tS_{G,S,\bar z}(T) \to [0,1]$.
  By Tietze's Extension Theorem
  this extends to a continuous function
  $\tS_{x,y,\bar z}(T) \to [0,1]$.
  For clarity we shall use $\varphi(yx,\bar z)$ to denote
  the corresponding definable predicate.

  Once this technical preliminary is taken care of we see that
  $\Stab(p)$ is defined by the following axiom scheme:
  \begin{gather*}
    \pi(y) =
    \Bigl\{
      \sup_{\bar z} | d_p\varphi(x,\bar z) - d_p\varphi(yx,\bar z) |
      = 0
    \Bigr\}_{\varphi \in \Phi_S}.
  \end{gather*}
  The moreover part easily follows.
\end{proof}

\begin{lem}
  Let $H < G$ be a type-definable subgroup of bounded index,
  say with parameters in $A$, and let $g \in H$.
  Then $g$ is generic over $A$ in $G$ if and only if it is
  generic over $A$ in $H$.
\end{lem}
\begin{proof}
  Naming $A$ in the language we may assume that $A = \emptyset$.
  Since $H$ has bounded index we may enumerate its cosets
  $\{ g_iH \}_{i<\lambda}$.
  Let $h_0 \in G$ be generic over $\{g_i\}_{i<\lambda}$.
  Then $h_0 \in g_iH$ for some $i$, and
  $h_1 = g_i^{-1}h_0 \in H$ is generic in $G$.
  Now let $h_2$ be generic in $H$.
  Without loss of generality we may assume that $h_1 \ind h_2$.
  Then $h_1h_2 \in H$ is generic both in $H$ and in $G$ and
  $h_1h_2 \ind h_1$.
  Thus $h_2 = h_1^{-1}h_1h_2$  is generic in $G$ as well.
  We have thus shown that every generic of $H$ is a generic of $G$.
  A similar argument shows that every generic of $G$ in $H$ is generic
  in $H$.
\end{proof}

\begin{prp}
  A type $p \in \tS_S(A)$ is generic if and only if $\Stab(p)$
  has bounded index in $G$.
\end{prp}
\begin{proof}
  There are only boundedly many generic types over $\fM$, since they
  do not fork over $\emptyset$ and therefore determined by their
  restriction to $\acl^{eq}(\emptyset)$.
  In addition, the action of $G$ on $\tS_S(\fM)$ restrict to an action
  of $G$ on the space of generic types, so the stabiliser of a generic
  type must be of bounded index.

  Conversely, assume $\Stab(p)$ has bounded index,
  and let $s \models p$.
  Then there exists $g \in \Stab(p)$ which is generic
  in $G$ over $A$, and we may further assume that
  $g \ind_A s$.
  Then $gs \models p$ is generic over $A$, i.e., $p$ is generic.
\end{proof}

Since $G$ acts transitively on the generic types over $\fM$, the
stabilisers of generic types are all conjugate.
It is also not difficult to check that if
$p \in \tS_S(\fM)$ is generic, $q \in \tS_G(\fM)$ is a generic type of
$\Stab(p)$ (and therefore of $G$), and
$s \models p \rest_{acl^{eq}(\emptyset)}$,
then $qs = p$.
If $q'$ is any other generic of $G$ then (since $G$ acts transitively
on its own generic types, on the left as well as on the right) there
exists $g \in G$ such that $q = q'g$
and $p = q'(gs)$.
Thus the right action of $S$ on $G$ send each and every generic type
of $G$ onto the generic types of $S$, complementing
\fref{lem:GenericTypeRightAction}.

\begin{thm}
  \label{thm:ConnectedComponent}
  Let $G$ be a type-definable group in a stable theory, say over
  $\emptyset$.
  Then $G$ admits a smallest type-definable group of bounded index
  (over any set of parameters), called the \emph{connected component}
  of $G$, and denoted $G^0$.
  It has the following additional properties.
  \begin{enumerate}
  \item The connected component $G^0$ is a normal subgroup of $G$,
    type-definable over $\emptyset$.
  \item The stabiliser of every generic type is equal to $G^0$.
  \item Each coset $gG^0$ contains a unique generic type over $\fM$.
  \item The generic type of $G^0$ is definable over $\emptyset$.
  \item If $p \in \tS_G(A)$ is any stationary generic type over
    a small set then $G^0 = \{g^{-1}h\colon g,h \models p\}$.
  \end{enumerate}
\end{thm}
\begin{proof}
  We start by constructing $G^0$ and proving the second item.
  Since left generic and right generic are the same,
  the action of $G$ on the generic types is transitive on
  either side.
  In particular, if $p,q\in \tS_G(\fM)$ are generic then there exists
  $g \in G$ such that $q = pg$,
  and thus $\Stab(p) = \Stab(q)$.
  Let this unique stabiliser of generic types be denoted $G^0$.
  Then $G^0$ is type-definable, and since $G$ $\emptyset$-invariant,
  so is $G^0$, and we may conclude that $G^0$ is type-definable over
  $\emptyset$ as well.
  We also already know that $G^0$ has bounded index in $G$.
  Assume now that $H \leq G^0$ is another type-definable subgroup
  of bounded index, say over $\emptyset$ (otherwise name the
  parameters in the language).
  Then there exists $p \in \tS_H(\fM)$ generic in $G$,
  so $\Stab(p) = G^0$, whereby $G^0 \subseteq H$.
  Thus $G^0$ is indeed the smallest type-definable subgroup of $G$ of
  bounded index.
  Notice that $G^0 \cap gG^0g^{-1}$ is also type-definable of bounded
  index for every $g \in G$, so $G^0$ is normal in $G$.
  This concludes the proof of the first two items.

  Let $p \in \tS_G(\fM)$ be generic in $G^0$.
  Since $G^0 = \Stab(p)$ acts transitively on its generic types,
  $p$ is the unique generic type in $G^0$.
  It follows that a coset $gG^0$ contains a unique generic type $gp$.
  The uniqueness of the generic type of $G^0$ implies that it is
  $\emptyset$-invariant, and therefore definable over $\emptyset$.

  Finally, let $p \in \tS_G(A)$ be any stationary generic type over a
  small set.
  Then $p\rest^\fM$ is the unique generic type in some coset $gG^0$.
  It follows that $gG^0$ is $A$-invariant, so
  $p \vdash x \in gG^0$.
  Thus $\{g^{-1}h\colon g,h \models p\} \subseteq G^0$.
  Conversely, let $g \in G^0$, and let $h \models p$, $g \ind_A h$.
  Since $G^0$ must also be the right-stabiliser of $p$
  we have $hg \models p$ as well, and $g = h^{-1}(hg)$, as desired.
\end{proof}
It follows that $G$ is connected (i.e., $G = G^0$) if and only if
it has a unique generic type.

\subsection{Global group ranks}

We have seen that a type of a member of $S$ is generic if and only if
the corresponding type in $\hat \fM$ is a non forking extension of the
unique type over $\emptyset$, i.e., if its $\hat \varphi$-type has the same
Cantor-Bendixson ranks as all of $S$ for every $\varphi \in \Phi_S$.
Thus the various $\varepsilon$-$\hat \varphi$-Cantor-Bendixson ranks play the role of
stratified local ranks characterising genericity.
In a superstable (and even more so in an $\aleph_0$-stable) theory one
would expect a similar characterisation via global Lascar and/or
Morley ranks.
We shall consider here the case of superstability and Lascar ranks.
Morley ranks are studied in a subsequent paper
\cite{BenYaacov:DefinabilityOfGroups}, and similar results are
proved.

Significant work regarding superstability has been carried out in the
context of metric Hausdorff cats, and in many cases definitions and
proofs transfer verbatim to continuous logic.
The objects of study in this context are partial types
constructed from complete types and conditions of the form
$d(x,y) \leq \varepsilon$ via conjunction and existential
quantification.
It is a general fact that if $\pi(\bar x,\bar y)$ is a partial type
then the property $\exists \bar y\, \pi(\bar x,\bar y)$, where the
existential quantifier is interpreted in a sufficiently saturated
model, is definable by a partial type as well.
This mostly happens in the following form.
Let $\pi(\bar x)$ be a partial type and $\varepsilon \geq 0$ a real
number.
We define $\pi(\bar x^\varepsilon)$ to be the partial type expressing
that
$\exists \bar y\,
\bigl( \pi(\bar y) \, \& \, d(\bar x,\bar y) \leq \varepsilon \bigr)$,
i.e., that $\pi$ is realised in the $\varepsilon$-neighbourhood of
$\bar x$.
For a tuple $\bar a$ we shall use $\bar a^\varepsilon$ as a notational
representation for the somewhat vague concept of
``$\bar a$ known up to distance $\varepsilon$'',
so in particular $\bar a^0$ is just another representation for
$\bar a$.
Accordingly, if $p(\bar x) = \tp(\bar a/C)$ then we define
$\tp(\bar a^\varepsilon/C) = p(\bar x^\varepsilon)$,
so in particular
$\tp(\bar a/C) = \tp(\bar a^0/C)
= \bigwedge_{\varepsilon > 0} \tp(\bar a^\varepsilon/C)$.

\begin{dfn}
  To an arbitrary theory $T$ we define $\kappa(T)$ to be the least
  infinite cardinal, if such exists, such that for every complete type
  $p(\bar x)$ over a set $A$, and for every $\varepsilon > 0$,
  there is a subset $A_0 \subseteq A$ such that
  $p(\bar x^\varepsilon)$ does not divide over $A_0$.

  We say that $T$ is \emph{simple} if $\kappa(T) \leq |\cL|^+$, and
  that it is supersimple if $\kappa(T) = \aleph_0$.
\end{dfn}

It follows from our earlier results that every stable theory is
simple.
It is true (but we shall not require it) that if $T$ is not simple
then $\kappa(T) = \infty$.

\begin{dfn}
  We say that $T$ is \emph{$\lambda$-stable} if for every set $A$,
  $|A| \leq \lambda$, the metric density character of $\tS_n(A)$ is at
  most $\lambda$.
  We define $\lambda_0(T)$ to be the least (infinite) cardinal of
  stability for $T$.
  We say that $T$ is \emph{superstable} if it is $\lambda$-stable for
  all $\lambda$ big enough.
\end{dfn}

In the context of Henson's logic for Banach space, this definition
dates back to Iovino \cite{Iovino:StableBanach}.
It was shown in \cite{BenYaacov-Usvyatsov:CFO} that $T$ is stable if
and only if it is $\lambda^{|\cL|}$-stable for all $\lambda$.
In particular, $T$ is stable if and only if
$\lambda_0(T) < \infty$, in which case $\lambda_0(T) \leq 2^{|\cL|}$.
The following is an example for a result whose statement and proof
translate word-for-word to the continuous logic setting.

\begin{fct}[{\cite[Theorem~4.13]{BenYaacov:Morley}}]
  A theory $T$ is $\lambda$-stable if and only if
  $\lambda^{<\kappa(T)} = \lambda \geq \lambda_0(T)$.
\end{fct}

\begin{cor}
  \label{cor:SuperstableSupersimple}
  A theory is superstable if and only if it is stable and
  supersimple.
\end{cor}

Many of the arguments that follow are valid both for stable and for
simple theories, and are stated as such, even though no development of
simplicity theory for continuous logic exists in the literature.
The reader may either refer to the development of simplicity in the
context of cats
\cite{BenYaacov:SimplicityInCats,BenYaacov:ThicknessAndCatTSF},
which encompasses that of continuous logic, or simply restrict his or
her attention to the stable case.

In light of \fref{cor:SuperstableSupersimple} and of the definition of
$\kappa(T)$ we wish to define global forking ranks $\SU_\varepsilon$
for $\varepsilon > 0$ such that
$\SU_\varepsilon(\bar a/C) > \SU_\varepsilon(\bar a/C\bar b)$
if and only if $\bar a$ ``$\varepsilon$-depends'' on $\bar b$ over
$C$.
We have quite a bit of liberty in choosing what
``$\varepsilon$-depends'' should mean
(i.e., different choices can give rise to ranks which have all the
properties we seek).
For example, we could say that this happens if
$\tp(\bar a^\varepsilon/C\bar b)$ forks over $C$, or if
$\tp(\bar a^\varepsilon/C\bar b) \wedge \tp(\bar a/\acl^{eq}(C))$
forks over $C$.
For consistency with earlier work we shall opt for a slightly more
complex definition which was given in
\cite{BenYaacov:SuperSimpleLovelyPairs} and which has some advantages
over other definitions for the purposes of that paper.
\begin{dfn}
  Let $\bar a$ and $\bar b$ be tuples, $C$ a set and
  $\varepsilon > 0$.
  We keep in mind that $\bar a^\varepsilon$ represents
  ``$\bar a$ known up to distance $\varepsilon$''.
  \begin{enumerate}
  \item We say that an indiscernible sequence
    $(\bar b_n)_{n \in \bN}$
    \emph{could be in $\tp(\bar b/\bar a^\varepsilon C)$}
    if there are $C'$ and a sequence $(\bar a_n)_n$
    such that $(\bar a_n\bar b_n)_n$ is $C'$-indiscernible,
    $\bar a_n\bar b_nC' \equiv \bar a\bar bC$
    and $d(\bar a_0,\bar a_1) \leq \varepsilon$.
  \item We say that $\bar a^\varepsilon \ind_C \bar b$ if every
    indiscernible sequence in $\tp(\bar b/C)$ could be in
    $\tp(\bar b/\bar a^\varepsilon C)$.
  \item We define $\SU_\varepsilon(\bar a/C)$ as may be expected:
    $\SU_\varepsilon(\bar a/C) \geq \alpha+ 1$ if and only if there is
    $\bar b$ such that $\bar a^\varepsilon \nind_C \bar b$
    and $\SU_\varepsilon(\bar a/C\bar b) \geq \alpha$.
    (In \cite{BenYaacov:SuperSimpleLovelyPairs} the notation
    $\SU(\bar a^\varepsilon/C)$ was used.)
  \end{enumerate}
\end{dfn}

A few justifications for these definitions may be in place.
Indeed, for $\varepsilon = 0$ it is easy to see that
a sequence $(\bar b_n)$ could be in $\tp(\bar b/\bar a^0C)$ if and
only if it admits a conjugate which is in $\tp(\bar b/\bar aC)$ in the
ordinary sense.
By a compactness argument, this is further equivalent to
the property that $(\bar b_n)$ could be in
$\tp(\bar b/\bar a^\varepsilon C)$ for every $\varepsilon > 0$.

Next, we wish to justify the notation $\ind$.
The result in \cite{BenYaacov:SuperSimpleLovelyPairs} on which this
justification relies does not immediately make sense in continuous
logic, since at some point it uses a local ranks argument which is
specific to the formalism of compact abstract theories.
The core of that argument resides in the following technical result,
which indeed can be proved very quickly in many different contexts
(simple or stable theories, classical logic, continuous logic, cats)
using the appropriate local ranks.
Since the notion of local forking ranks varies drastically between
contexts we shall provide here a more combinatorial and therefore more
universal argument.

\begin{lem}
  \label{lem:AvoidLocalRank}
  Let $a$, $b$ and $c$ be three tuples, possibly infinite,
  in a stable or even simple theory
  (classical, continuous, or any other setting in which basic
  simplicity or stability hold).
  Assume moreover that $b \equiv_a c$.
  Then $a \ind_b c$ if and only if $a \ind_c b$.
\end{lem}
\begin{proof}
  Assume not, say $a \ind_c b$ but $a \nind_b c$.
  We construct by induction a sequence
  $(b_n)_{n\in \bN}$ such that
  $b_nb_{n+1} \equiv_a bc$
  and $ab_{n+1} \ind_{b_n} b_{<n}$
  for all $n$.
  We start with $b_0 = b$, $b_1 = c$.
  Then for each $n$ we can choose $b_{n+2}$ such that
  $b_{n+1}b_{n+2} \equiv_a bc$ and
  $b_{n+2} \ind_{ab_{n+1}} b_{\leq n}$.
  Our induction hypothesis tells us that
  $ab_{n+1} \ind_{b_n} b_{<n}$
  and $a \ind_{b_{n+1}} b_n$.
  By standard independence calculus we obtain
  $a \ind_{b_{n+1}} b_{\leq n}$,
  whence
  $ab_{n+2} \ind_{b_{n+1}} b_{\leq n}$, as desired.

  Since $a \nind_b c$,
  there exists a $b$-indiscernible sequence $(c^k)^k$, starting
  with $c^0 = c$, such that there exists no $a'$ satisfying
  $a'c^k \equiv_b ac$ for all $k$.
  For each $n$ we may choose a copy $(c_n^k)_k$
  such that $(c_n^k)_k,b_{n+1},b_n \equiv (c^k)_k,c,b$,
  so in particular this copy starts with
  $c_n^0 = b_{n+1}$ and is indiscernible over $b_n$.
  Since $b_{n+1} \ind_{b_n} b_{<n}$, we may choose $(c_n^k)_k$ to
  be indiscernible over $b_{\leq n}$.

  Let us now define $c_n = b_{n+1}$ for all $n$
  and consider the sequence $(b_nc_n)_n$.
  Applying compactness, we can find for arbitrarily big $\lambda$ a
  sequence $(b_ic_i)_{i<\lambda}$ in $\tp(bc/a)$,
  such that for each $i$ there exists a sequence
  $(c_i^k)_k$ which is indiscernible over $b_{\leq i}c_{<i}$
  such that $(c_i^k)_k,c_i,b_i \equiv (c^k)_k,c,b$,
  so in particular $c_i^0 = c_i$.
  Each of these sequences witnesses that
  $a \nind_{b_{<i}c_{<i}} b_ic_i$, contradicting the local character.
\end{proof}

We can now prove:
\begin{fct}[{\cite[Lemma~1.13]{BenYaacov:SuperSimpleLovelyPairs}}]
  Assume that $T$ is simple (or stable).
  For $\bar a^\varepsilon$, $C$ and $\bar b$,
  the following imply one another from top to bottom:
  \begin{enumerate}
  \item $\bar a^\varepsilon \ind_C \bar b$.
  \item There is a Morley sequence for $\bar b$ over $C$ which could
    be in $\tp(\bar b/\bar a^\varepsilon C)$.
  \item $\tp(\bar a^\varepsilon/\bar bC)$ does not fork over $C$.
  \item $\bar a^{2\varepsilon} \ind_C \bar b$.
  \end{enumerate}
\end{fct}
\begin{proof}
  We only need to provide a proof for (ii) $\Longrightarrow$ (iii),
  for which the original proof used local ranks.

  Indeed let $(\bar a_n)_n$ and $C'$ witness that
  a Morley sequence $(\bar b_n)$ could be in
  $\tp(\bar b/\bar a^\varepsilon C)$.
  We may assume that $(\bar a_n\bar b_n)_n$ is indiscernible over
  $CC'$, and we may further assume that for some $\tilde a$,
  the pair $\tilde a\bar b$ continues the sequence
  $(\bar a_n\bar b_n)_n$ indiscernibly over $CC'$.
  Standard arguments regarding indiscernibility provide that
  $\bar b \ind_{(\bar b_n)_n} CC'a_0$.
  Since $\bar b \ind_C (\bar b_n)_n$ we obtain
  $\bar b \ind_C C'\bar a_0$.
  By \fref{lem:AvoidLocalRank}
  we have
  $\bar b \ind_{C'} C$ and thus
  $\bar b \ind_{C'} \bar a_0$.
  Let $f$ be an automorphism sending $\tilde a\bar bC'$ to
  $\bar a\bar bC$.
  Then $\bar b \ind_C f(\bar a_0)$ and
  $d(f(\bar a_0),\bar a) = d(\bar a_0,\tilde a)\leq \varepsilon$.
  Thus $\tp(\bar a^\varepsilon/\bar bC)$ does not fork over $C$,
  as desired.
\end{proof}

Thus in particular $\bar a^0 \ind_C \bar b$ if and only if
$\tp(\bar a/\bar bC)$ does not fork over $C$, i.e., if and only if
$\bar a \ind_C \bar b$, justifying our notation.
By earlier observations, this is further equivalent to
$\bar a^\varepsilon \ind_C \bar b$ holding for all $\varepsilon > 0$.
It is further shown in \cite{BenYaacov:SuperSimpleLovelyPairs}
that $T$ is supersimple if and only
if $\SU_\varepsilon(\bar a/B)$ is ordinal for every finite tuple $\bar a$ and
$\varepsilon > 0$.
Moreover, in a supersimple theory $T$, $\SU_\varepsilon$ ranks
characterise independence: $\bar a \ind_C B$ if and only if
$\SU_\varepsilon(\bar a/C) = \SU_\varepsilon(\bar a/BC)$ for all
$\varepsilon > 0$.

Since our global ranks depend (inevitably) on a metric resolution
parameter $\varepsilon$ we may only hope to characterise genericity in
case the metric is invariant under the group action, i.e., if the
action of each $g \in G$ on $S$ is an isometry.

We have seen that if $g$ is generic over $s,A$ then $gs$ is generic
over $A$.
We now prove a converse:
\begin{lem}
  Assume $\langle G,S\rangle$ is a type-definable transitive group
  action in a stable theory $T$,
  $s \in S$ generic over a set $A$,
  $t \in S$ satisfying $t \ind_A s$.
  Then there is $g \in G$, $g \ind_A t$ such that $gs = t$.
  Moreover, $g$ can be chosen generic over $A$ (i.e., over $At$).
\end{lem}
\begin{proof}
  We may assume $A = \emptyset$.
  First choose $g \in G$ generic, $g \ind s,t$.
  Then $s$ is generic over $g,t$,
  so $gs \ind g,t$
  By standard independence calculus we obtain $g \ind gs,t$.
  Since the action is transitive we can find $h \in G$ such that
  $hgs = t$, and we may take it so that $h \ind_{gs,t} g$.
  Then $g$ is generic over $t,gs,h$, and so is
  $gh$, and in particular $hg \ind t$.
  Then $g' = hg$ is
  generic over $t$ as required.
\end{proof}

\begin{thm}
  Assume $\langle G,S\rangle$ is a type-definable transitive group action
  with an invariant metric in a superstable continuous
  theory $T$, $p \in \tS_S(A)$.
  Then $p$ is generic if and
  only if
  $\SU_\varepsilon(p)
  = \SU_\varepsilon(S)
  = \sup \{ \SU_\varepsilon(q)\colon q \in \tS_S(\emptyset) \}$
  for all $\varepsilon > 0$.
  In particular, types of maximal $\SU_\varepsilon$-rank exist.
\end{thm}
\begin{proof}
  We may assume that $A = \emptyset$.
  We shall use the fact that if $p \in \tS_n(B)$, $q \in \tS_m(B)$ and
  $f\colon p(\fM) \to q(\fM)$ is $B$-definable and isometric then
  $\SU_\varepsilon(p) = \SU_\varepsilon(q)$ for all $\varepsilon > 0$.
  The proof of this fact is left as an exercise
  to the reader.

  Let $s \models p$, and assume first that $p$ is generic.
  Let $t \in S$ realise an arbitrary type over $\emptyset$.
  We may nonetheless assume that $t \ind s$.
  By the Lemma there exists $g \ind t$ such that
  $gs = t$.
  Since multiplication by $g$ is isometric we obtain
  $\SU_\varepsilon(s) \geq \SU_\varepsilon(s/g)
  =\SU_\varepsilon(t/g) = \SU_\varepsilon(t) = \SU_\varepsilon(q)$.

  Conversely, let $s \in S$ and assume that
  $\SU_\varepsilon(s) \geq \SU_\varepsilon(q)$
  for all $q \in \tS_S(\emptyset)$ and all $\varepsilon > 0$.
  Let $g \in G$, $g \ind_A s$.
  Then
  $\SU_\varepsilon(gs/g) = \SU_\varepsilon(s/g) = \SU_\varepsilon(s)
  \geq \SU_\varepsilon(gs) \geq \SU_\varepsilon(gs/g)$.
  Thus equality holds all the way for all $\varepsilon > 0$, whereby
  $gs \ind g$, so $s$ is generic.
\end{proof}

\providecommand{\bysame}{\leavevmode\hbox to3em{\hrulefill}\thinspace}


\begin{thebibliography}{{Ben}03b}

\bibitem[{Ben}a]{BenYaacov:DefinabilityOfGroups}
Itaï {Ben Yaacov},
  \href{http://math.univ-lyon1.fr/~begnac/articles/DefOSGrp.pdf}
  {\emph{Definability of groups in $\aleph_0$-stable metric structures}},
  Journal of Symbolic Logic, to appear,
  \href{http://arxiv.org/abs/0802.4286}{arXiv:0802.4286}.

\bibitem[{Ben}b]{BenYaacov:RandomVariables}
\bysame, \href{http://math.univ-lyon1.fr/~begnac/articles/RandVar.pdf}
  {\emph{On theories of random variables}}, submitted,
  \href{http://arxiv.org/abs/0901.1584}{arXiv:0901.1584}.

\bibitem[{Ben}03a]{BenYaacov:SimplicityInCats}
\bysame, \href{http://dx.doi.org/10.1142/S0219061303000297} {\emph{Simplicity
  in compact abstract theories}}, Journal of Mathematical Logic \textbf{3}
  (2003), no.~2, 163--191.

\bibitem[{Ben}03b]{BenYaacov:ThicknessAndCatTSF}
\bysame, \href{http://dx.doi.org/10.4064/fm179-3-2} {\emph{Thickness, and a
  categoric view of type-space functors}}, Fundamenta Mathematic{\ae}
  \textbf{179} (2003), 199--224.

\bibitem[{Ben}05]{BenYaacov:Morley}
\bysame, \href{http://dx.doi.org/10.2178/jsl/1122038916} {\emph{Uncountable
  dense categoricity in cats}}, Journal of Symbolic Logic \textbf{70} (2005),
  no.~3, 829--860.

\bibitem[{Ben}06]{BenYaacov:SuperSimpleLovelyPairs}
\bysame, \href{http://dx.doi.org/10.2178/jsl/1154698575} {\emph{On
  supersimplicity and lovely pairs of cats}}, Journal of Symbolic Logic
  \textbf{71} (2006), no.~3, 763--776,
  \href{http://arxiv.org/abs/0902.0118}{arXiv:0902.0118}.

\bibitem[{Ben}09]{BenYaacov:NakanoSpaces}
\bysame, \href{http://dx.doi.org/10.4115/jla.2009.1.1} {\emph{Modular
  functionals and perturbations of {N}akano spaces}}, Journal of Logic and
  Analysis \textbf{1:1} (2009), 1--42,
  \href{http://arxiv.org/abs/0802.4285}{arXiv:0802.4285}.

\bibitem[BU]{BenYaacov-Usvyatsov:CFO}
Itaï {Ben Yaacov} and Alexander Usvyatsov,
  \href{http://math.univ-lyon1.fr/~begnac/articles/cfo.pdf} {\emph{Continuous
  first order logic and local stability}}, Transactions of the American
  Mathematical Society, to appear,
  \href{http://arxiv.org/abs/0801.4303}{arXiv:0801.4303}.

\bibitem[Iov99]{Iovino:StableBanach}
José Iovino, \emph{Stable {B}anach spaces and {B}anach space structures, {I}
  and {II}}, Models, algebras, and proofs (Bogotá, 1995), Lecture Notes in
  Pure and Appl. Math., vol. 203, Dekker, New York, 1999, pp.~77--117.

\bibitem[Pil96]{Pillay:GeometricStability}
Anand Pillay, \emph{Geometric stability theory}, Oxford Logic Guides, vol.~32,
  The Clarendon Press Oxford University Press, New York, 1996, Oxford Science
  Publications.

\bibitem[Poi85]{Poizat:Cours}
Bruno Poizat, \emph{Cours de théorie des modèles}, Nur al-Mantiq
  wal-Ma'rifah, 1985.

\end{thebibliography}
\end{document}